\documentclass[10pt]{amsart}

\usepackage{amsmath, amscd, amssymb}
\usepackage[frame,cmtip,arrow,matrix,line,graph,curve]{xy}
\usepackage{graphpap, color}
\usepackage[mathscr]{eucal}

\numberwithin{equation}{section}

\newcommand{\CC}{\mathbb{C}}

\newcommand{\PP}{\mathbb{P}}
\newcommand{\QQ}{\mathbb{Q}}
\newcommand{\RR}{\mathbb{R}}
\newcommand{\ZZ}{\mathbb{Z}}

\newcommand{\bk}{\mathbf{k}}

\newcommand{\bq}{\mathbf{q}}

\newcommand{\kk}{\bk}
\def\bA{\mathbf{A}}

\newcommand{\cal}{\mathcal}

\def\cA{{\cal A}}

\def\cC{{\cal C}}
\def\cD{{\cal D}}
\def\cE{{\cal E}}

\def\cK{{\cal K}}
\def\cL{{\cal L}}
\def\cM{{\cal M}}
\def\cN{{\cal N}}
\def\cO{{\cal O}}

\def\cR{{\cal R}}
\def\cT{{\cal T}}
\def\cU{{\cal U}}

\def\cX{{\cal X}}

\def\cZ{{\cal Z}}

\def\mapright#1{\,\smash{\mathop{\lra}\limits^{#1}}\,}

\let\cMM=\cM
\let\cNN=\cN

\def\BM{H^{BM}\lsta}

\def\dual{^{\vee}}

\def\sta{^\ast}

\def\virt{^{\mathrm{vir}}}
\def\upmo{^{-1}}
\def\sta{^{\ast}}
\def\dpri{^{\prime\prime}}

\def\pri{^{\prime}}

\def\mm{{\mathfrak m}}
\def\sta{^*}

\def\lra{\longrightarrow}

\def\lsta{_{\ast}}

\newcommand{\si}{\sigma}

\def\begeq{\begin{equation}}
\def\endeq{\end{equation}}
\def\and{\quad{\rm and}\quad}
\def\bl{\bigl(}
\def\br{\bigr)}
\def\defeq{:=}
\def\mh{\!:\!}
\def\sub{\subset}
\def\Ao{{\mathbb A}^{\!1}}

\def\Po{{\mathbb P^1}}
\def\and{\quad\text{and}\quad}
\def\mapright#1{\,\smash{\mathop{\lra}\limits^{#1}}\,}

\def\llra{\mathop{-\!\!\!\lra}}

\def\reg{_{\text{reg}}}

\DeclareMathOperator{\Tor}{Tor} \DeclareMathOperator{\Ext}{Ext}
 \DeclareMathOperator{\ext}{\cE
\it{xt}} \DeclareMathOperator{\Hom}{Hom}
 \DeclareMathOperator{\Aut}{Aut}
\DeclareMathOperator{\image}{Im}

 \DeclareMathOperator{\rank}{rank}
\DeclareMathOperator{\spec}{Spec}

\newtheorem{prop}{Proposition}[section]
\newtheorem{theo}[prop]{Theorem}
\newtheorem{lemm}[prop]{Lemma}
\newtheorem{coro}[prop]{Corollary}

\newtheorem{defi}[prop]{Definition}
\newtheorem{conj}[prop]{Conjecture}

\newtheorem{defiprop}[prop]{Definition-Proposition}
\newtheorem{subl}[prop]{Sublemma}

\DeclareMathOperator{\coker}{coker}

\def\llra{\,\mathop{-\!\!\!\lra}\,}

\def\bone{{\mathbf 1}}
\def\mgn{\cM_{g,n}}

\def\mxn{\cM_{\chi,n}}

\def\Ob{\cO b}

\def\kzz{\kk[\![z]\!]}

\def\loc{_{\mathrm{loc}}}

\def\mcn{\cM_{\chi,n}}

\def\bul{^\bullet}

\def\ev{\text{ev}}

\def\aa{\alpha}

\def\lab#1{\label{#1}[{#1}]\  }
\let\lab=\label

\title{Gromov-Witten invariants of varieties with holomorphic 2-forms}
\date{}

\author{Young-Hoon Kiem}
\address{Department of Mathematics and Research Institute
of Mathematics, Seoul National University, Seoul 151-747, Korea}
\email{kiem@math.snu.ac.kr}

\author{Jun Li}
\address{Department of Mathematics, Stanford University, Stanford,
USA} \email{jli@math.stanford.edu}

\begin{document}

\begin{abstract}
We show that a holomorphic two-form $\theta$ on a smooth
{algebraic} variety $X$ localizes the virtual fundamental class of
the moduli of stable maps $\mgn(X,\beta)$ to the locus where
$\theta$ degenerates; it then enables us to define {the}
\emph{localized GW-invariant}, an algebro-geometric analogue of the
local invariant of Lee and Parker in symplectic geometry
\cite{Lee-Parker}, which coincides with the ordinary GW-invariant
when $X$ is proper. It is deformation invariant. {Using this, we
prove formulas for low degree GW-invariants of minimal general type
surfaces with $p_g>0$ conjectured by Maulik and Pandharipande.}
\end{abstract}
\maketitle 

\section{Introduction}
In recent years, Gromov-Witten invariant (GW-invariant, for short)
has played an important role in research in algebraic geometry and
in Super-String theories. Various effective techniques have been
contrived to compute these invariants, such as localization by torus
action \cite{GraberPand}, degeneration method \cite{IonelParker,
Li2,LiRuan}, quantum Riemann-Roch and Lefschetz
\cite{CoatesGivental}, to name a few. For curves, these invariants
have been completely determined in \cite{OkounkovPand}. For higher
dimensional case, much 
{more remains to be done.} In this paper, we provide a new
technique, called the \emph{localization by holomorphic two-form},
that is the algebro-geometric analogue of what was first 
{discovered} by Lee and Parker \cite{Lee-Parker} in symplectic
geometry.

The localization by holomorphic two-form is a localization theorem
on virtual cycles. As is known, central to the construction of
GW-invariants is to replace the fundamental class of the moduli of
stable maps by its {\sl virtual fundamental cycle}.
In this paper, we show that a holomorphic two-form $\theta\in
H^0(\Omega^2_X)$ on a smooth quasi-projective variety $X$ localizes
(or forces) the virtual fundamental cycle to (support in) the locus
of those stable maps $f\mh C\to X$ whose images lie in the
degeneracy locus of $\theta$.

Based on this localization by holomorphic two-form, for a pair of a
smooth quasi-projective variety and a holomorphic two-form
$(X,\theta)$ with certain properness requirement, we shall define
the so-called \emph{localized GW-invariant}. This invariant is
deformation invariant; when the degeneracy locus of $\theta$ is
smooth, it is equivalent to the localized invariant of its normal
bundle; when $X$ is proper, it coincides with the ordinary
GW-invariant.

Applying this to a smooth minimal general type surface $S$ with
positive $p_g=h^2(\cO_S)$, we 
{conjecture} that its GW-invariants are obtained from the localized
GW-invariants of the total space of a theta characteristic of a
smooth curve $D$ of genus $K_S^2+1$. In the ideal case, the
localized GW-invariants are explicitly related to the twisted
GW-invariants of $D$, which includes all GW-invariants of $S$
\emph{without} descendant insertions.
{This conjecture was partially proved by} Lee-Parker for the case
when $S$ has smooth canonical divisors \cite{Lee-Parker}. For
others, we prove a degeneration formula that allows us to reduce the
complete set of GW-invariants to the twisted GW-invariants of curves
and of some low degree relative localized invariants. For the case
of degrees 1 and 2, we work out their details and verify the
formulas conjectured by Maulik and Pandharipande \cite[(8) and
(9)]{MaulikPand}.


This localized invariant has recently been employed by W-P. Li and
the second author to study the GW-invariants of the Hilbert schemes
of points on surfaces \cite{LL}.

We now provide a more detailed outline of this paper. In section 2,
we work out the details of the localization by a holomorphic
two-form and the localized GW-invariants. Let $X$ be a smooth
quasi-projective variety equipped with a (nontrivial) holomorphic
two-form $\theta\in H^0(\Omega^2_X)$. Then $\theta$ gives rise to a
homomorphism, called a \emph{cosection}
\[
\sigma:\cO b_\cM\lra \cO_\cM
\]
from the obstruction sheaf of the moduli space $\cM=\mgn(X,\beta)$
of stable maps to its structure sheaf.
Let $Z(\sigma)$ be the locus over which $\sigma$ is not surjective.
We then show that
we can canonically construct a \emph{localized} virtual fundamental
class
\[
[\cM]\virt_{\mathrm{loc}}\in H\lsta^{BM} (Z(\sigma))
\]
satisfying the following properties: it is deformation invariant
under a mild condition on the cosections; {\sl its pushforward to
$H\lsta(\cM)$ coincides with the ordinary virtual class $[\cM]\virt$
in case the degeneracy locus $Z(\sigma)$ is proper.}

This immediately leads to

\begin{theo} For a pair $(X,\theta)$ of a smooth projective variety
and a holomorphic two-form, the virtual fundamental class of
$\mgn(X,\beta)$ vanishes unless $\beta$ is represented by a
$\theta$-null stable map.
\end{theo}

Here we say a stable map $f\mh C\to X$ with fundamental class
$\beta=f_*[C]$ is $\theta$-null if the image of $df$ lies in the
degeneracy locus of $\theta$. (See section 3 for details.)

This theorem recovers the vanishing results of J. Lee and T. Parker
\cite{Lee-Parker}. For instance, if $X$ is $2m$-dimensional and
$D=zero(\theta^m)$, then all the GW-invariants vanish unless $\beta$
is in the image of $H_2(D,\ZZ)$ in $H_2(X,\ZZ)$ by the inclusion
map. In particular, if $X$ is a holomorphic symplectic manifold
(i.e. $D=\emptyset$), all GW-invariants vanish.

In section 3, we define the localized GW-invariants as the integral
over the localized virtual fundamental class of tautological
classes. By construction, when $X$ is proper, the localized
GW-invariants coincide with the ordinary GW-invariants. The
localized invariants are invariant under a class of deformation
relevant to our study.

For a smooth minimal general type surface $S$ with $p_g>0$, we shall
take $X$ to be the total space of any theta characteristic $L$ on a
smooth projective curve $D$ of genus $h=K_S^2+1$ satisfying
$h^0(L)\equiv\chi(\cO_S)\mod (2)$. It is easy to see that $X$ has a
holomorphic two-form non-degenerate away from the zero section of
$X\to D$.

\begin{conj}
Let $S$ be a smooth {minimal general type surface with positive
$p_g>0$}. Its GW-invariants $\langle \cdots \rangle _{\beta,g}^S$
vanish unless $\beta$ is a non-negative integral multiple of
$c_1(K_S)$. In case $\beta=d c_1(K_S)$ for an integer $d> 0$, we let
$(D,L)$ be a pair of a smooth projective curve of genus $K_S^2+1$
and its theta characteristic with parity $\chi(\cO_S)$, and let $X$
be the total space of $L$. Then there is a canonical homomorphism
$\rho\mh H^{\ast}(S,\ZZ)\to H^{\ast}(X,\ZZ)$ so that for any classes
$\gamma_i\in H^{\ast}(S,\ZZ)$ and integers $\alpha_i\geq 0$,
$i=1,\cdots,n$,
\[
\langle\tau_{\alpha_1}(\gamma_1)\cdots\tau_{\alpha_n}(\gamma_n)\rangle^S_{\beta,g}
=\langle\tau_{\alpha_1}(\rho(\gamma_1))\cdots
\tau_{\alpha_n}(\rho(\gamma_n))\rangle^{X,\mathrm{loc}}_{d[D],g}.
\]
\end{conj}

This conjecture is proved by Lee-Parker \cite{Lee-Parker} when $S$
has a smooth canonical divisor. In \S \ref{subs3.3}, we provide a
different proof by showing invariance of the localized virtual
fundamental class under deformation of $S$ to the normal bundle $X$
of a smooth canonical curve. In general, by studying the deformation
of the (analytic) germ of a canonical divisor, we have confirmed the
conjecture for a wider class of surfaces, including the case when
$S$ has a reduced canonical divisor. We shall address this in
\cite{KL3}.

In light of this Conjecture, the GW-invariants of minimal general
type surfaces are reduced to the localized GW-invariants of the
total space $X$ of a theta-characteristic $L$ over a smooth curve
$D$.
Let $(\pi,f)$ represent the universal stable map to $D$. In case
$R^1\pi_*f^*L$ is locally free, we prove that the localized
GW-invariants of $X$ are the twisted GW-invariants of $D$ with
appropriate sign modifications (Proposition \ref{fundtwist}). This
in particular enables us to recover the following result of
Lee-Parker \cite{Lee-Parker}.

\begin{theo} Let $X$ be the total space of a theta-characteristic
$L$ on a smooth curve $D$. Then the localized genus $g$ GW-invariant
of $X$ with homology class $d[D]$ is
\[
\langle 1\rangle^{X,\mathrm{loc}}_{d,g} =
\sum_{u:d\text{-fold \'etale cover of }D}
\frac{(-1)^{h^0(u^*L)}}{|\mathrm{Aut}(u)|}
\]
for $g=d(h-1)+1$ where $h$ is the genus of $D$. All the other
localized invariants are zero.
\end{theo}

\smallskip

For the GW-invariants \emph{with} descendant insertions, we reduce
the problem to a simpler one by using a degeneration formula for the
(relative) localized GW-invariants by degenerating $X$ into a union
of $Y_1\cong X$ and $Y_2=\PP^1\times\CC$. By working out the
relative localized invariants for low degree, in section 4 we prove

\begin{theo} Let $X\to D$ be a theta characteristic over a smooth
curve of genus $h$. Let $\gamma\in H^2(D,\ZZ)$ be the Poincar\'e
dual of a point in $D$. Then the degree one and two
GW-invariants\footnote{{See \S \ref{section4.1} for the definition
of GW-invariants with not necessarily connected domains.}} with
descendants are
\begin{equation*}\langle
\prod_{i=1}^n\tau_{\aa_i}(\gamma)\rangle^{X,\bullet}_{[D],\mathrm{loc}}=
(-1)^{h^0(L)}\prod_{i=1}^n\frac{\aa_i!}{(2\aa_i+1)!}(-2)^{-\aa_i};
\end{equation*}
\begin{equation*}\langle
\prod_{i=1}^n\tau_{\aa_i}(\gamma)\rangle^{X,\bullet}_{2[D],\mathrm{loc}}
= (-1)^{h^0(L)}\,
2^{h+n-1}\prod_{i=1}^n\frac{\aa_i!}{(2\aa_i+1)!}(-2)^{\aa_i}.
\end{equation*}
\end{theo}

The above two formulas are conjectured by Maulik and Pandharipande
\cite{MaulikPand}.

\smallskip

In section 5, we discuss a few special cases where our localization
argument may enable us to compute the GW-invariants of three-folds.
For instance, if $X$ is a $\PP^1$-bundle over a surface with
$|K_S|\ne \emptyset$, we can combine the virtual localization by
torus action (\cite{GraberPand}) and our localization by holomorphic
two-form, to reduce the computation of GW-invariants of $X$ to the
curve case where the degeneration method \cite{Li2, OkounkovPand}
applies effectively.

\smallskip
\noindent{\bf Acknowledgment}: The first author is grateful to the
Stanford Mathematics department for support and hospitality while he
was visiting during the academic year 2005/2006. The second author
thanks D. Maulik for sharing with him his computation for an example
that is crucial for the second part of the paper; he also thanks E.
Ionel for stimulating discussions. We thank J. Lee and T. Parker for
stimulating questions and for pointing out several oversights in our
previous draft.

\section{Localizing virtual cycles by cosections of obstruction sheaves}

In this section, we shall show that for a Deligne-Mumford stack with
a perfect-obstruction theory, a meromorphic cosection of its
obstruction sheaf will localize its virtual cycle to the degeneracy
locus of the cosection.

More precisely, we let $M$ be a Deligne-Mumford stack endowed with a
perfect obstruction theory and let $\Ob$ be its obstruction sheaf
\cite{Li-Tian,BF}. We call $\sigma$ a meromorphic \emph{cosection}
of $\Ob$ if there is a dense open subset $U\sub M$ so that $\sigma$
is a sheaf homomorphism
\begin{equation}
\sigma: \Ob|_U\llra \cO_U.
\end{equation}
We define its degeneracy locus as the union
$$Z(\sigma)=(M-U)\cup \{s\in U\mid \sigma(s)=0: \Ob\otimes k(s)\lra
 k(s)\}.
$$

The main result of this section is

\begin{lemm}[Localization Lemma]\lab{main-lemma}
Let $M$ be a Deligne-Mumford stack endowed with a perfect
obstruction theory and suppose the obstruction sheaf admits a
meromorphic cosection $\sigma$. Then there is a canonical cycle
$$[M]\virt\loc\in \BM\bl Z(\sigma)\br
$$
whose image under the obvious $i\lsta\mh \BM\bl Z(\sigma)\br\to
\BM(M)$ is the virtual cycle
$$[M]\virt=i\lsta[M]\virt\loc \in \BM(M).
$$
\end{lemm}

Here by a perfect obstruction theory we mean either the one defined
by Tian and the second author in \cite{Li-Tian} using relative
obstruction theory or by Behrend and Fantechi in \cite{BF} using
cotangent complex. As pointed out by Kresch \cite{Kresch}, the two
constructions are equivalent and produce identical virtual cycles.

Further, the localized cycle $[M]\virt\loc$ is deformation invariant
under a technical condition.


\subsection{Virtual normal cones and cosections}

To prove this lemma, we shall first prove that the virtual normal
cone associated to the obstruction theory of $M$ lies in the kernel
cone of the cosection. We begin with proving a fact about the normal
cone, which was essentially proved in \cite{BGP}.

\begin{lemm}\lab{j2.1}
Let $W\sub V$ be a closed subscheme of a smooth scheme $V$ defined
by the vanishing $s=0$ of a section $s$ of a vector bundle $E$ on
$V$; let $C_WV $ be the normal cone to $W$ in $V$, embedded in $E$
via the section $s$. Suppose the cokernel
\[
\cA=\coker\{ds: \cO_W(T_V)\lra \cO_W(E)\}
\]
admits a surjective sheaf homomorphism $\sigma\mh\cA\to\cO_W$. Then
the cone $C_WV $ lies entirely in the subbundle $F\sub E$ that is
the kernel of the composite
\begin{equation}\lab{j2.4}
\cO_W(E)\mapright{\phi}\cA\mapright{\sigma} \cO_W.
\end{equation}
\end{lemm}

\begin{proof}
To prove the lemma, we shall view $C_WV\sub E$ as the specialization
of the section $t^{-1}s\sub E$ as $t\to 0$. More precisely, we
consider the subscheme
$$\Gamma=\{(t^{-1}s(w),t)\in E\times (\Ao-0)\mid w\in V,\ t\in\Ao-0\}.
$$
For $t\in \Ao-0$, the fiber $\Gamma_t$ of $\Gamma$ over $t\in\Ao$ is
merely the section $t^{-1}s$ of $E$. We let $\bar\Gamma$ be the
closure of $\Gamma$ in $E\times \Ao$. By definition, the central
fiber $\bar\Gamma\times_{\Ao}0\sub E$ is the normal cone $C_WV $
alluded before. Clearly, $C_WV $ is of pure dimension $\dim V$.

Now let $N\sub C_WV $ be any irreducible component and let
$\alpha\in N$ be a general closed point of $N$. Suppose $\alpha$
does not lie in the zero section of $E$. Then we can find a regular
irreducible curve $C$ and a morphism $\iota\mh C\to \bar\Gamma$ that
passes through $\alpha$, say $\iota(0)=\alpha$. We let
$$\eta: C\lra V\and \zeta:  C\lra \Ao
$$
be the projections induced by those from $E\times \Ao$ to $V$ and to
$\Ao$; we require that $\zeta$ dominates $\Ao$. We then choose a
uniformizing parameter $\xi$ of $C$ at $0$ so that
$\zeta\sta(t)=\xi^n$ for some $n$. Because $\iota(0)=\alpha$,
$\xi^{-n} s\circ \eta(\xi)$ specializes to $\alpha$; hence
$s\circ\eta$ has the expression
$$s\circ\eta=\alpha\xi^n+O(\xi^{n+1}),
$$
and thus $\eta\upmo(W)=\spec\CC[\xi]/(\xi^n)$. In particular,
pulling back the exact sequence
$$\cO_W(T_V)\mapright{ds} \cO_W(E)\lra \cA\lra 0
$$
via the induced morphism
$$\bar\eta=\eta|_{\eta\upmo(W)}:\bar\eta\upmo(W)=\spec\CC[\xi]/(\xi^n)
\lra W,
$$
we obtain
\begin{equation}\lab{j2.3}
\bar\eta\sta(T_V) \mapright{\bar\eta\sta(ds)} \bar\eta\sta(E)
\mapright{\bar\eta\sta(\phi)} \bar\eta\sta(\cA)\lra 0
\end{equation}
in which $\bar\eta\sta(\phi)$ is the pullback of $\phi$ in
(\ref{j2.4}) and
$$\bar\eta\sta(ds)=d(\alpha\xi^n+O(\xi^{n+1}))\equiv n\alpha\xi^{n-1}d\xi\mod \xi^n.
$$

On the other hand, because \eqref{j2.3} is exact, the composition
\begin{equation}\lab{2.4}
\bar\eta\sta(\sigma)\circ \bar\eta\sta(\phi)\circ\bar\eta\sta(ds)=0.
\end{equation}
Now suppose $(\sigma\circ\phi)(\alpha)\ne0$. Then the vanishing
(\ref{2.4}) implies that $\xi^{n-1}=0\in \CC[\xi]/(\xi^n)$, a
contradiction. Therefore, $\alpha$ lies in the kernel of
\eqref{j2.4}, and so does the cone $C_WV$. This proves the lemma.
\end{proof}

In case $\sigma$ is a meromorphic cosection of $E$ and $U$ is the
largest open subset over which $\sigma$ is defined and surjective,
we define the (cone) kernel $E(\sigma)$ of $\sigma$ to be the union
of the restriction to $M-U$ of $E$ with the kernel of $\sigma|_{U}$:
$$E(\sigma)=E|_{M-U}\cup \mathrm{Ker}(\sigma\mh E|_U\lra\cO_U).
$$

\begin{coro} Let $W\sub V$ be as in Lemma \ref{main-lemma} except
that $\sigma$ is only assumed to be a meromorphic cosection. Then
the cone $C_WV$ lies entirely in the kernel $E(\sigma)$.
\end{coro}

In the following, we shall construct a localized Gysin map
by intersecting with smooth sections that ``almost'' split the
cosection $\sigma$.

\subsection{Localized Gysin maps}

Let $\sigma: E\to \cO_M$ be a meromorphic cosection of the vector
bundle $E$ over a quasi-projective complex scheme $M$ with
$Z(\sigma)\sub M$ and $E(\sigma)$ its degeneracy locus and its
kernel cone. For the bundle $E\to M$ we recall that the topological
Gysin map
\[
s_E^!:A_*E\lra A\lsta M
\]
is defined by intersecting any cycle $W\in Z\lsta E$ with the zero
section $s_E$ of $E$. To define the localized Gysin homomorphism, we
shall use smooth section of $E$ that almost lifts $1\in
\Gamma(\cO_M)$.

For this, we first pick a splitting of $\sigma$ away from the
degeneracy locus $Z(\sigma)$. Because $\si$ is surjective away from
$Z(\sigma)$, possibly by picking a hermitian metric on $E$
we can find a smooth section $\check{\si}\in C^\infty\bl
E|_{M-D(\si)}\br$ so that $\si\circ\check{\si}=1$.

Next, we pick a sufficiently small (analytic) neighborhood $\cU$ of
$Z(\sigma)\sub M$ which is properly homotopy equivalent to
$Z(\sigma)$. Because $M$ is quasi-projective, such a neighborhood
$\cU$ always exists. We then extend $\check \sigma|_{M-\cU}$ to a
smooth section $\check{\si}_{ex} \in C^\infty\bl E\br$ and pick a
smooth function $\rho\mh M\to \RR^{>0}$ so that
$$\xi=\rho\cdot\check\sigma_{ex}\in C^\infty\bl E\br
$$
is a small perturbation of the zero section of $E$.

Now let $W\subset E(\si)$ be any closed subvariety. By fixing a
stratification of $W$ and of $M$ by complex subvarieties, we can
choose the extension $\check{\si}_{ex}$ and the function $\rho$ so
that the section $\xi$ intersects $W$ transversely. As a
consequence, the intersection $W\cap \xi$, which is of pure
dimension, has no real codimension 1 strata. Henceforth, it defines
a closed oriented Borel-Moore chain cycle in $M$. But on the other
hand, since $\sigma\circ\xi |_{M-\cU} =\rho\in C^\infty\bl
M-\cU\br$, $\xi$ is disjoint from $W$ over $M-\cU$. Thus
$W\cap\xi\sub E|_\cU$ is a closed Borel-Moore cycle in $E|_\cU$. In
this way, under the projection $\pi_\cU\mh E|_\cU\to\cU$, we obtain
a class
$$[\pi_\cU(W\cap \xi)]\in \BM(\cU).
$$
Applying the standard transversality argument, one easily shows that
this class is independent of the choice of $\cU$ and the section
$\xi$; thus it only depends on the cycle $W$ we begin with.
Furthermore, because $\cU$ is properly homotopy equivalent to
$Z(\sigma)$, $\BM(Z(\sigma))\cong \BM(\cU)$. Therefore, the newly
constructed class can be viewed as a class in $\BM(Z(\sigma))$.

\begin{defiprop}\lab{2.5}
We define the localized Gysin map
$$s_{E,\mathrm{loc}} ^!: Z\lsta E(\sigma)\lra \BM(Z(\sigma))
$$
to be the linear map that sends any subvariety $W\sub Z\lsta
E(\sigma)$ to the cycle $[\pi_\cU(W\cap\xi)]\in \BM(Z(\sigma))$. It
sends any two rationally equivalent cycles in $Z\lsta E(\sigma)$ to
the same homology class in $\BM(Z(\sigma))$; therefore it factors
through a homomorphism from the group of cycle classes:
$$s_{E,\mathrm{loc}} ^!: A\lsta E(\sigma)\lra \BM(Z(\sigma)).
$$
\end{defiprop}

\begin{proof}
The proof is standard and shall be omitted.
\end{proof}

Next, we shall investigate the case for a DM-stack with
perfect-obstruction theory and a cosection of its obstruction sheaf.

\subsection{DM-stack with perfect obstruction theory}

Let $M$ be a DM-stack with perfect obstruction theory and with a
cosection $\sigma\mh\Ob_M\to\cO_M$ of its obstruction sheaf; let $E$
be a vector bundle on $M$ whose sheaf of sections $\cE$ surjects
onto $\cO_M$; by the construction of virtual cycle, the obstruction
theory of $M$ provides a unique cone cycle $W\in Z\lsta E$, the
virtual normal cone of $M$.

\begin{lemm}\lab{ker-cone}
Let $\tilde\sigma\mh \cE\to \cO_{M}$ be the composite of $\sigma$
with the quotient homomorphism $\cE\to\Ob_{M}$. Then the (virtual)
normal cone $W\in Z\lsta E$ lies in the (cone) kernel
$E(\tilde\sigma)$ of $\tilde\sigma$.
\end{lemm}

\begin{proof}
Let $U\sub M$ be the largest open subset over which $\sigma\mh
\Ob_{M}\to \cO_M$, and hence the composite $\tilde\sigma\mh
\cE\to\cO_M$, is surjective. We let $F\sub E|_U$ be the kernel
subbundle of $\tilde\sigma$. To prove the lemma we only need to show
that the restriction of the cone $W$ over $U$ is entirely contained
in $F$. Since this is a local property, we only need to prove this
over every closed point $p\in U$. And by replacing a neighborhood of
$p\in U$ by its \'etale covering, we can assume without loss of
generality that $M$ is a scheme.

To proceed, we recall the construction of the cone $W$ at $p$. We
let $\hat M$ be the formal completion of $M$ at $p$ with $\rho\mh
\hat M\to M$ the tautological morphism; we let $E_{p}$ be the fiber
of $E$ at $p$; let $n$ and $(z)$ be
$$n=\rank
E+\text{vir}.\dim M,\quad (z)=(z_1,\cdots,z_n).
$$
Then the obstruction theory of $M$ at $p$ provides a Kuranishi map
$f\in \kk[\![z]\!]\otimes E_{p}$ so that $\hat M$ is isomorphic to
the subscheme $(f=0)\sub {\hat V}=\spec\kk[\![z]\!]$. We claim that
we can choose $f$ and two isomorphisms $\psi_1$ and $\psi_2$ as
shown so that the cokernel $\coker(df)$ fits into the commutative
diagram
\begin{equation}\lab{comd2.4}
\begin{CD} \cO_{\hat M}(T_{\hat V})@>{df}>>
\cO_{\hat M}\otimes E_{p} @>{\text{pr}}>>\coker(df) @>>> \,\cO_{\hat
M}\\
@. @V{\cong}V{\psi_1}V @V{\cong}V{\psi_2}V @|\\
@.\rho\sta\cE @>{}>> \rho\sta\Ob_{M} @>{\rho\sta(\sigma)}>> \
\cO_{\hat M}.
\end{CD}
\end{equation}

Before we prove the claim, we shall see how this leads to the proof
of Lemma. Once we have this diagram, then since $\rho(p)\in U$, the
pullback sheaf homomorphism $\rho\sta(\sigma)\circ \psi_2$ is
surjective. Hence by the proof of Lemma \ref{j2.1}, the fiber of the
normal cone $C_{\hat M}{\hat V}$ over $p$ is entirely contained in
the kernel vector space:
$$C_{\hat M}{\hat V}\times_{\hat M} p\sub
F_{p}=\mathrm{ker}\{\rho\sta(\sigma)\circ \psi_2\circ
\text{pr}|_{p}: E_{p}\lra \CC\}.
$$
On the other hand, under the isomorphism $\psi_1$ the vector space
$F_{p}$ is isomorphic to the restriction to $p$ of the kernel vector
bundle $F=\mathrm{ker}\{\tilde\sigma|_U: \cE|_U\lra\cO_U\}$, and the
fiber of the cone $C_{\hat M} {\hat V}$ over $p$ is identical to the
fiber over $p$ of the cone $W$ \cite[Lemma 3.3]{Li-Tian}. Therefore,
$W|_p\sub F|_p$. Since $p$ is arbitrary, this proves that the
support of $W$ over $U$ is entirely contained in the subbundle
$F\sub E|_U$.

We now prove the claim. Indeed, the existence of $\psi_1$ and
$\psi_2$ follows from the definition of the perfect obstruction
theory based on the cotangent complex of $M$ phrased in \cite{BF}.
In case we use the perfect obstruction theory phrased in
\cite{Li-Tian}, their existence follows directly from the proof of
Lemma 2.5 in \cite{Li-Tian}, once the following technical
requirement is met\footnote{This requirement is the consistency of
the properties (1) and (3) for $k=1$ in \cite[Lemma 2.5]{Li-Tian}.
It was used but not checked. Here we provide the details of it.}.

To state and verify this requirement, we let $p\in M$ be as before
(and $M$ is a scheme as assumed) and let $S\sub M$ be an affine
neighborhood of $p$. Roughly speaking, the requirement is that we
can find a complex of locally free sheaves
$[\cE_1\mapright{\sigma}\cE_2]$ so that in addition to that the
cokernel of $\sigma$ and its dual $\sigma\dual$ are $\Ob_S$ and
$\Omega_S$, respectively, the homomorphism $\sigma$ defines the
obstruction to first order extensions.


We now set up more notation. We embed $S$ as a closed subscheme of a
smooth affine scheme $V$ of $\dim V=\dim T_p S$. We let
$A=\Gamma(\cO_S)$, let $B=\Gamma(\cO_V)$, let $\iota\sta\mh B\to A$
be the quotient homomorphism with $I\sub B$ the ideal
$I=\iota^{\ast-1}(0)$; thus $B/I=A$. We then consider the trivial
(ring) extension of $A$ by $\Omega_B\otimes_BA$: $A_2=A\oplus
\Omega_B\otimes_BA$; we let $d\mh B\to\Omega_B$ be the differential
and consider the homomorphism $\xi\mh B\to A_2$ defined via
$b\mapsto(\iota\sta b,db)$. Then by taking $J\sub A_2$ the ideal
generated by $\xi(I)$ and letting $A_1=A_2/J$, the homomorphism
$\xi$ descends to $\xi_1\mh A\to A_1$. Since $J^2=0$, the triple
$(A_1,A_2,\xi_1)$ associates to an obstruction class
$$o\in Ob_A\otimes_{A}J,\quad Ob_A=\Gamma(S,\Ob_M),
$$
to lifting $\xi_1$ to $A\to A_2$. Now let $E_2$ be a free $A$-module
making $Ob_A$ its quotient module. We lift $o$ to an $\hat o\in
E_2\otimes_A J$, which via $E_2\otimes_A J\sub E_2\otimes_B\Omega_B$
defines a homomorphism
$$\psi: E_2\dual\lra \Omega_B\otimes_BA.
$$
We claim that the cokernel of $\psi$ and $\psi\dual$ are $\Omega_A$
and $Ob_A$, respectively. After this, the complex
$[\Omega_B\dual\otimes_BA\to E_2]$ will satisfy the requirement for
Lemma 2.5 in \cite{Li-Tian} for $k=1$; its inductive proof provides
us with the Kuranishi map we seek for.

We now prove the claim. We let $R$ be the cokernel of $\psi$ and
shall prove that as quotient sheaves of $\Omega_B\otimes_BA$,
$R=\Omega_A$. First, by viewing $J$ as an $A$-module, it is a
submodule of $\Omega_B\otimes_BA$ satisfying
$\Omega_B\otimes_BA/J=\Omega_A$. Because $\psi$ is defined by an
element in $E_2\otimes_A J$, $\psi(E_2\dual)\sub J$. Hence
$\Omega_A$ is canonically a quotient sheaf of $R$, say via $\tau\mh
R\to\Omega_A$. To show that $\tau$ is an isomorphism, we let
$T=A\oplus R$, let $T_0=A\oplus \Omega_A$, let $K=\ker(\tau)$, and
let $f\mh A\to A\oplus \Omega_A$ be the tautological homomorphism
defined via $a\mapsto (a,da)$. Then the triple $(T,T_0,f)$
associates to an obstruction class $\bar o\in Ob_A\otimes_A K$ to
lifting $f$ to $A\to A\oplus R$. However, by the base change
property, $\bar o$ is the image of $o$ under $Ob_A\otimes_A J\to
Ob_A\otimes_A K$.  Since $\hat o$ is a lift of $o\in Ob_A\otimes_A
J$, by the commutativity
$$\begin{CD}
E_2\otimes_A J @>>> E_2\otimes_A K\\
@VVV @VVV\\
Ob_A\otimes_A J @>>> Ob_A\otimes_A K
\end{CD}
$$
and by the fact that the element $\hat o$ has vanishing image under
$E_2\otimes_B\Omega_B\to E_2\otimes_A R$, we see immediately that
$\bar o=0$. Thus $f$ lifts to $A\to A\oplus R$. By the property of
the cotangent module, this lifting is given by a homomorphism of
$A$-modules $\varphi\mh \Omega_A\to R$ so that the composite
$\tau\circ \varphi\mh \Omega_A\to\Omega_A$ is the homomorphism
associated to $f\mh A\to A\oplus \Omega_A$, thus
$\tau\circ\varphi=id_A$. It follows that $\Omega_A$ is a direct
summand of $R$. But for $\mm\sub A$ the maximal ideal of $p\in S$,
we have $R\otimes_A A/\mm=\Omega_A\otimes A/\mm$; thus possibly
after shrinking $S$, $\Omega_A=R$. This proves that $R=\Omega_A$.

It remains to show that $\psi\dual\mh\Omega_B\dual\otimes_BA\to E_2$
has cokernel $Ob_A$. The proof is similar. We first show that the
composite $\Omega_B\dual\otimes_BA\to E_2\to Ob_A$ is trivial. Let
$T_0\pri=A$ and let $T_1\pri=A\oplus \Omega_B\otimes_BA$. The
identity $id\mh A\to T_0\pri$ obviously lifts to $A\to T_1\pri$.
Thus the obstruction $\tilde o$ to lifting $id$ to $A\to T_1\pri$ is
trivial. But by the base change property, $\tilde o$ is the image of
$o\in Ob_A\otimes_A J$ under $Ob_A\otimes_A J\to Ob_A\otimes_B
\Omega_B$, which also is the composite $\Omega_B\dual\otimes_BA\to
E_2\to Ob_A$. Thus this composite vanishes, which proves that as
quotient modules of $E_2$, $\coker(\psi\dual)$ surjects onto $Ob_A$.

To prove $\coker(\psi\dual)=Ob_A$, we shall use the property that
$S$ has a perfect obstruction theory. Namely, there is a free
$A$-module $E_1$ and a homomorphism $\eta\mh E_1\to E_2$ such that
$\coker(\eta)=\Ob_A$ and $\coker(\eta\dual)=\Omega_A$. We now let
$Q_1=\image(\psi\dual)$ and $Q_2=\image(\eta)$, both as submodules
of $E_2$. Because $\coker(\psi\dual)\cong\coker(\eta\dual)$, for
every integer $m$, $Q_1\otimes_A A/\mm^m$ and $Q_2\otimes_A A/\mm^m$
have the same dimension as vector spaces; the same holds true for
$\Tor^1(Q_1,A/\mm^m)$ and $\Tor^1(Q_2,A/\mm^m)$. Thus
$$\dim_{\CC}Ob_A\otimes
A/\mm^n=\dim_{\CC}\coker(\psi\dual)\otimes_A A/\mm^n;
$$
since one is the quotient of the other, this implies that
$Ob_A\otimes_A A/\mm^n=\coker(\psi\dual)\otimes_A A/\mm^n$ for all
$m$. Hence after shrinking $S$ if necessary,
$\coker(\psi\dual)=Ob_A$. This completes the proof of the Lemma.
\end{proof}

With this lemma, we are ready to construct the localized virtual
cycle
$$[M]\virt\loc\in \BM(Z(\sigma)).
$$
We first write $W=\sum m_i W_i$ as the weighted sum of (reduced)
irreducible closed substacks $W_i\sub E$. To each such $W_i$, we let
$M_i$ be the image stack of $W_i\to M$ and pick a quasi-projective
$Z_i$ together with a proper and generically finite $\rho_i: Z_i\lra
M_i$. We then pull back the pair $W_i\sub E|_{M_i}$ and the
cosection:
$$\tilde W_i=W_i\times_{M}Z_i\sub \tilde E_i=E\times_{M}Z_i,
\quad\tilde\sigma_i=\rho_i\sta(\tilde\sigma): \tilde E_i\lra
\cO_{Z_i}.
$$
By the previous lemma, the cycle $\tilde W_i$ lies in the cone
kernel $\tilde E_i(\tilde\sigma_i)$. Hence we can apply the
localized Gysin map to the class $[\tilde W_i]$ to obtain
$$s_{\tilde E_i,\mathrm{loc}} ^![\tilde W_i]\in \BM(Z(\tilde \sigma_i)).
$$

We now let $\eta_i\mh Z(\tilde \sigma_i)\to Z(\sigma)$ be the
induced map. Because $\rho_i$ is proper, $\eta_i$ is also proper.
Thus it induces a homomorphism of Borel-Moore homology
$$\eta_{i\ast}: \BM(Z(\tilde\sigma_i))\lra
\BM(Z(\sigma)).
$$
Finally, we let $\deg(\rho_i)$ be the degree of $\rho_i$, and define
the localized virtual cycle
$$[M]\virt\loc=\sum \frac{m_i}{\deg(\rho_i)} \eta_{i\ast}\bl
s_{\tilde E_i,\mathrm{loc}} ^![\tilde W_i]\br \in \BM(Z(\sigma)).
$$

\subsection{Deformation invariance of the localized virtual cycles}

Like the ordinary virtual cycle, the localized virtual cycle is
expected to remain constant under deformation of complex structures.
In the following, we shall prove this under a technical assumption.

We let $t\in T$ be a pointed smooth affine curve; let $\pi\mh M\to
T$ be a DM-stack over $T$ with a perfect obstruction theory and
obstruction sheaf $\Ob_M$; we suppose $M$ has a perfect relative
obstruction theory with relative obstruction sheaf $\Ob_{M/T}$, as
defined in \cite{Li-Tian}. By definition, $\Ob_M$ and $\Ob_{M/T}$
fits into the exact sequence
\begin{equation}\lab{2.6}
\lra\cO_M\lra
\Ob_{M/T} \lra \Ob_{M}\lra 0.
\end{equation}
Note that the restriction of $\Ob_{M/T}$ to each fiber
$M_t=M\times_T t$ is the obstruction sheaf $\Ob_{M_t}$ of $M_t$.

We now suppose there is a cosection
$$
\sigma: \Ob_{M/T}\lra \cO_{M}.
$$
We let $Z(\sigma)$ be the union of $Z(\sigma_t)\sub M_t$ for all
$t\in T$.

\begin{prop} \lab{constancy} Suppose $Z(\sigma)$ is proper over $T$;
suppose the cosection $\sigma$ lifts to a cosection $\sigma\pri\mh
\Ob_{{M}}\to\cO_{{M}}$. Then the localized virtual cycles
$[{M}_t]\virt\loc$ are constant in $t$ as classes in
$H\lsta(Z(\sigma))$.
\end{prop}

\begin{proof}
We shall prove the proposition by showing that applying the
localized Gysin map to the rational equivalence used in deriving the
deformation invariance of the ordinary virtual cycles will provide
us with the homologous relation necessary for the constancy of the
classes $[{M}_t]\virt\loc\in H\lsta(Z(\sigma))$.

Without loss of generality, we can assume that $M$ is a
quasi-projective scheme and $E$ is a vector bundle on $M$ whose
sheaf of sections $\cE=\cO_M(E)$ makes $\Ob_M$ its quotient
sheaf\footnote{The general case can be treated using the technique
developed in \cite{Li2}.}. Then according to the construction of
virtual cycles, the obstruction theory of ${M}$ provides us with a
unique cone cycle $W\in Z\lsta E$ whose intersection with the zero
section of $E$ is the virtual cycle of ${M}$.

For us, we shall use the localized Gysin map to derive a localized
virtual cycle of ${M}$. We let $\tilde\sigma\pri\mh \cE\to\cO_{{M}}$
be the composite of $\cE\to\Ob_{{M}}$ with $\sigma\pri$, and let
$E(\tilde\sigma\pri)$ be the kernel cone of $\tilde\sigma\pri$. Then
Lemma \ref{ker-cone} tells us that $W$ is a cycle in $Z\lsta
E(\tilde\sigma\pri)$. Because $\Ob_{{M}/T}\to\Ob_{{M}}$ is
surjective, $Z(\sigma\pri)=Z(\sigma)$, which by assumption is proper
over $T$. Thus by applying the localized Gysin map we obtain a
homology class
$$s_{E,\mathrm{loc}} ^!([W])\in \BM(Z(\sigma)).
$$

Further, because $T$ is smooth, for each closed point $t\in T$ the
inclusion $t\hookrightarrow T$
defines a Gysin homomorphism
\[t^!:\BM(Z(\sigma))\lra
H\lsta(Z(\sigma_t)).
\]
Here the image lies in the ordinary homology group because
$Z(\sigma)$ is proper over $T$.

By the elementary property of Gysin homomorphism,
if we let $\iota_t\mh Z(\sigma_t)\to Z(\sigma)$ be the inclusion,
then the classes
$$\iota_{t\ast}\left( t^!\bl
s_{E,\mathrm{loc}}^!([W])\br\right) \in H\lsta(Z(\sigma))
$$
are constant in $t$. Therefore, to prove the proposition we only
need to prove that
\begin{equation}\lab{identity}t^!\bl s_{E,\mathrm{loc}}^!([W])\br=
[{M}_t]\virt\loc \in
H^{}\lsta(Z(\sigma_t)).
\end{equation}

Ordinarily, we prove this by applying the Gysin map to a rational
equivalence between the cycle $W$ and the cycle used to define the
virtual cycle $[{M}_t]\virt$. In our case, to get an identity in
ordinary homology classes, we need to make sure that the rational
equivalence lies inside a cone so that Proposition \ref{2.5} can be
applied.

To this end, a quick review of the construction of the rational
equivalence is in order. First, using the homomorphism
$\cO_{{M}_t}\to\Ob_{{M}_t}$ induced by the first arrow in
(\ref{2.6}), and possibly after replacing $\cE\to \Ob_M$ by another
quoient sheaf, we can write $\Ob_{{M}_t}$ as a quotient sheaf of
$\cO_{{M}_t}\oplus\cE_t$ for $\cE_t=\cE\otimes_{\cO_M}\cO_{M_t}$.
Then the composition
$$\cO_{{M}_t}\oplus\cE_t
\mapright{\phi}\Ob_{{M}_t}\mapright{\sigma_t}\cO_{{M}_t}
$$
defines a cosection $\zeta_t$ of $\cO_{{M}_t}\oplus\cE_t$. Because
$\sigma$ lifts to $\sigma\pri$, $\cO_{{M}_t}$ lies in its kernel.
Hence, over ${M}_t-Z(\sigma_t)$ the kernel of $\zeta_t$ is a direct
sum of the trivial line bundle $\bone$ on ${M}_t$ with
$F_t=F|_{{M}_t}$, the restriction of the kernel $F$ of $\sigma\mh
E\to\cO_{M}$. As before, we write $(\bone\oplus E_t)(\zeta_t)$ for
the kernel cone of $\zeta_t$.

With the quotient sheaf $\cO_{{M}_t}\oplus\cE_t\to\Ob_{{M}_t}$, the
obstruction theory of $M_t$ provides us a unique cone cycle
$W_t\pri\sub \bone\oplus E_t$ whose support, according to Lemma
\ref{ker-cone}, lies in $(\bone\oplus E_t)(\zeta_t)$. Because
$Z(\sigma_t)$ is proper, the localized Gysin map then defines the
localized virtual cycle
$$[M
_t]\virt\loc=s_{\bone\oplus E_t,\mathrm{loc}} ^![W_t\pri]\in
H_\ast(Z(\sigma_t)).
$$

The cycle $W_t\pri$, up to a factor $\Ao$, is rationally equivalent
to the normal cone $C_{W\times_Tt}W\sub E_t\oplus \bone$. As argued
in \cite[page 146]{Li-Tian}, the rational equivalence lies in the
total space of the vector bundle $L\defeq\bone\oplus\bone\oplus
E_t$. We let
\begin{equation}\lab{L}
\pi_{23}: L\to \bone\oplus E_t\and \pi_{13}: L\lra \bone\oplus E_t,
\end{equation}
be the projections to the indicated factors of $L=\bone\oplus
\bone\oplus E_t$. We let $B_1=\pi_{23}\sta W_t\pri$; let
$B_2=\pi_{13}\sta C_{W\times_Tt}W$; let
$$\tilde\zeta:
L\mapright{pr_3} E_t\mapright{\tilde\sigma\pri} \cO_{{M}_t}
$$
and let $L(\tilde\zeta)$ be the kernel cone of the homomorphism
$\tilde\zeta$. Clearly, both $B_1$ and $B_2\in Z\lsta
L(\tilde\zeta)$, and
$$s_{L,\mathrm{loc}} ^![B_1]=[{M}_t]\virt\loc\and
s_{L,\mathrm{loc}} ^![B_2]=t^![{M}]\virt\loc
$$
On the other hand, as shown in \cite[page 146]{Li-Tian} (see also
\cite{BF,Kresch}) there is a canonical rational equivalence $Q\in
R\lsta L$ (here $R\lsta L$ is the set of rational equivalence in
$L$) so that $\partial Q=B_1-B_2$. We claim that
\begin{equation}\lab{rat}
\text{Supp} (Q)\sub L(\tilde\zeta).
\end{equation}
Note that once this is proved, then $B_1$ and $B_2$ are rationally
equivalent cycles in $Z\lsta L(\tilde\zeta)$. By Proposition
\ref{2.5}, we will have
$$s_{L,\mathrm{loc}} ^![B_1]=s_{L,\mathrm{loc}} ^![B_2]\in
H\lsta(Z(\sigma)).
$$

We now prove (\ref{rat}). We let $p\in{M}_t-Z(\sigma_t)$ be any
closed point. Following the notation introduced in the proof of
Lemma \ref{ker-cone} we let $f\in\kk[\![z]\!]\otimes E_{p}$ be a
Kuranishi map of ${M}$ at $p$. Then $\hat M=(f=0)\sub \hat V=\spec
\kzz$ and the fiber over $p$ of the cone $W$ is identical to the
closed fiber of the normal cone $C_{\hat M}\hat V$ to $\hat M$ in
$\hat V$.

The fiber $W_t\pri|_p$ has a similar description. We pick an $\tilde
h\in\cO_T$ so that $\tilde h\upmo(0)=\{t\}$ and pick a lifting
$h\in\cO_{\hat V}$ of $\tilde h$. Then as shown in \cite[page
144]{Li-Tian}, the pair $(h,f)\in \kzz\otimes(\CC\oplus E_p)$ is a
Kuranishi map of $M_t$ at $p$. Hence the formal completion $\hat
M_t$ of $M_t$ at $p$ is isomorphic to $(h=f=0)\sub \hat
V_t=\spec\kzz/(h)$. Likewise, the fiber at $p$ of the cone $W_t\pri$
is identical to the closed fiber of the normal cone $C_{\hat
M_t}\hat V_t$.

Because $f$ is a Kuranishi map of $M$ at $p$, as argued in
(\ref{comd2.4}), we have an isomorphism $\coker(df)\cong
\rho\sta\Ob_M$ and the surjective composite homomorphism
\begin{equation}\lab{split}
\cO_{\hat M}\otimes E_p\lra \coker(df)\ \cong\ \rho\sta \Ob_M
\mapright{\rho\sta(\sigma\pri)} \cO_{\hat M}.
\end{equation}
We let $\cK_0\sub \cO_{\hat M}\otimes E_p$ be the kernel (subsheaf)
of (\ref{split}) and let $\cK_1\sub \cO_{\hat M}\otimes E_p$ be its
complementary subsheaf:
$$\cO_{\hat M}\otimes E_p=\cK_0\oplus\cK_1.
$$
Because $\cO_{\hat M}\otimes E_p$ is locally free, we can extend
this decomposition to that over $\hat V$,
\begin{equation}\lab{dec2}
\cO_{\hat V}\otimes E_p=\cK_0\oplus\cK_1
\end{equation}
still labeled as $\cK_0$ and $\cK_1$. Because $\kzz\otimes
E_p=\Gamma(\cO_{\hat V}\otimes E_p)$, $f$ can be viewed as a section
of $\cO_{\hat V}\otimes E_p$, thus decomposed as $(f_0,f_1)$ under
the decomposition (\ref{dec2}).

\begin{subl}
The subscheme $\hat M_{red}\sub \hat V$, which is $\hat M$ endowed
with the reduced scheme structure, is identical to the subscheme
$\hat M\pri_{red}=(f_0=0)_{red}\sub \hat V$; the support of the
normal cone $C_{\hat M}\hat V$ to $\hat M$ in $\hat V$ coincides
with the support of the normal cone $C_{\hat M\pri}\hat V$, both as
subcones in $\hat M\times E_p$. The same conclusion holds with $\hat
M\sub \hat V$ replaced by $\hat M_t\sub\hat V_t$ and $\hat M\pri$
replaced by $\hat M_t\pri=(h=f_0=0)\sub\hat V_t$.
\end{subl}

We shall prove the Sublemma momentarily.

Granting the Sublemma, we now prove the Proposition. To begin with,
we let $\hat K_0$ and $\hat K_1$ be the associated vector bundles on
$\hat M_t$ of the sheaf $\cK_0$ and $\cK_1$, let $\hat\bone$ be the
trivial line bundle on $\hat M_t$, and let $\hat L$ be the direct
sum bundle $\hat\bone\oplus\hat\bone\oplus (\hat K_0\oplus\hat
K_1)$. As before, we let
$$\hat\pi_{13}: \hat L\lra \bone\oplus(\hat
K_0\oplus\hat K_1);\quad \hat\pi_{23}: \hat L\lra \hat\bone\oplus
(\hat K_0\oplus\hat K_1)
$$
be the projections similar to those in (\ref{L}). Then by working
with subschemes $M_t$ and $\hat M\sub \hat V$, we obtain a rational
equivalence $\hat Q$ in $\hat L$ so that
$$\partial \hat Q=\hat\pi_{23}\sta C_{\hat M_t}\hat V_t-\hat\pi_{13}\sta
C_{C_{\hat M}\hat V\times_T t}C_{\hat M}\hat V.
$$
Working with the subschemes $\hat M\pri_t$ and $\hat M\pri\sub \hat
V$ instead, we obtain another cone $\hat Q\pri$ in $\hat L$ so that
$$\partial \hat Q\pri=\hat\pi_{23}\sta C_{\hat M_t\pri}\hat V_t-\hat\pi_{13}\sta
C_{C_{\hat M\pri}\hat V\times_T t}C_{\hat M\pri}\hat V.
$$
However, since $\hat M\pri$ is defined by the vanishing of $f_0$,
which is a section of $\cK_0$ over $\hat V$, the support of the
rational equivalence $\hat Q\pri$ naturally lies in $\hat
\bone\oplus\hat\bone\oplus \hat K_0$. On the other hand, following
the explicit construction of the rational equivalence by Kresch
\cite{Kresch}, the support of the rational equivalence $\hat Q$ only
depends on the supports of the subschemes $\hat M$ and $\hat M_t$,
and on the supports of the cones $C_{\hat M_t}\hat V_t$ and
$C_{C_{\hat M}\hat V\times_T t}C_{\hat M}\hat V$. Therefore, though
$\hat Q$ may be different from $\hat Q\pri$, their supports are
identical. Consequently, the support of $\hat Q$ lies entirely in
$\hat \bone\oplus\hat\bone\oplus \hat K_0$.

Finally, because the fiber over $p$ of the rational equivalence $Q$
is identical to that of $\hat Q$, we conclude that the support of
$Q$ over $M-Z(\sigma_t)$ lies inside the subbundle
$\bone\oplus\bone\oplus F$. This proves the proposition.
\end{proof}

\begin{proof}[Proof of Sublemma]
We first show that $\hat M_{red}=\hat M\pri_{red}$. Suppose not.
Then we can find a map $\rho\mh \kk[\![t]\!]\to\hat M$ that does not
factor through $\hat M\pri\sub\hat M$. Because $\hat M$ and $\hat
M\pri$ are defined by $f_0=f_1=0$ and by $f_0=0$ respectively, we
must have $f_0\circ\rho\equiv 0$ while $f_1\circ\rho\not\equiv 0$.
Hence for some $n$ and $a\ne 0\in \cK_1\otimes k(\bar p)$, $f_1\circ
\rho(t)\equiv at^n\mod t^{n+1}$. Thus by the proof of Lemma
\ref{j2.1}, $a\in C_{\hat M}\hat V$, violating the conclusion of
Lemma \ref{j2.1} that $C_{\hat M}\hat V\sub \hat K_0$. For the same
reason, the proof of Lemma \ref{j2.1} shows that $C_{\hat M}\hat
V\sub \hat K_0$ forces that as sets, $C_{\hat M}\hat V\times_{\hat
M}\bar p=C_{\hat M\pri}\hat V\times_{\hat M}\bar p$.

The same argument works for the case $\hat M_t$ and $\hat
M_t\pri\sub \hat V_t$. This proves the Sublemma.
\end{proof}

\section{Localized GW-invariants}

We continue to let $X$ be a smooth quasi-projective variety endowed
with a holomorphic two-form $\theta$. This form induces a cosection
of the obstruction sheaf $\Ob_{\cM}$ of the moduli space
$\mgn(X,\beta)$. We assume that the degeneracy locus $Z(\sigma)$ of
$\sigma$ is proper. In this case, we can pair the localized virtual
cycle with the tautological classes on $\cM$ to define the localized
GW-invariants of $(X,\theta)$.

In this section, we shall also prove the deformation invariance of
such invariants under some technical condition; we shall derive a
formula of the localized invariants in a special case, sufficient to
recover all GW-invariants of surfaces without descendants, first
proved by Lee-Parker \cite{Lee-Parker}.

\subsection{Cosection of the obstruction sheaf}

Let $X$ be a smooth complex quasi-projective variety endowed with a
nontrivial holomorphic two-form $\theta\in \Gamma(\Omega_X^2)$. The
two-form $\theta$ can be viewed as an anti-symmetric homomorphism
\begin{equation}\lab{4.1}
\hat{\theta}:T_X\lra \Omega_X,\quad (\hat\theta(v),v)=0
\end{equation}
from the tangent bundle to the cotangent bundle. Such a two-form
will induce a cosection of the obstruction sheaf of the moduli space
$\mgn(X,\beta)$. For convenience, in case the data $g$, $n$, $X$ and
$\beta$ are understood implicitly, we shall abbreviate
$\mgn(X,\beta)$ to $\cM$.

Let $S$ be an open subset of $\cM$; let $f\mh \cC\to X$ and $\pi\mh
\cC\to S$ be the universal family over $S$; let
\begin{equation} \lab{yh2.2}
R^1\pi_*f^*T_X\lra R^1\pi_*f^*\Omega_X\lra
R^1\pi_*\Omega_{\cC/S}\lra R^1\pi_*\omega_{\cC/S},
\end{equation}
be the sequence of homomorphisms in which the first arrow is induced
by (\ref{4.1}) and the second is induced by
$f\sta\Omega_X\to\Omega_{\cC/S}$. Because
$R^1\pi_*\omega_{\cC/S}\cong \cO_S$, this sequence provides us with
a canonical homomorphism
\begin{equation}\lab{yh2.7} \sigma_f:R^1\pi_*f^*T_X\lra
 \cO_S.
\end{equation}

\begin{lemm}\lab{lem-yh2.3}
The composition \begeq \lab{yh2.9} \cE
xt^1_\pi(\Omega_{\cC/S},\cO_{\cC}) \lra R^1\pi_*f^*T_X \lra \cO_S,
\endeq
in which the first arrow is induced by $f^*\Omega_X\to
\Omega_{\cC/S}$ and the second arrow by $\sigma_f$, is a trivial
homomorphism.
\end{lemm}

\begin{proof}
The natural homomorphism $f^*\Omega_X\to \Omega_{\cC/S}$ coupled
with $\hat\theta$ induces a sequence of homomorphisms and their
composite
$$
\Theta:\Omega_{\cC/S}^\vee\lra f^*T_X\mapright{\hat{\theta}}
f^*\Omega_X \lra\Omega_{\cC/S}.
$$
Because $\hat\theta$ is anti-symmetric, this composite is also
anti-symmetric
$$(\Theta(v),v)=0, \quad\forall v\in \Omega_{\cC/S}\dual.
$$
Because $\Omega_{\cC/S}$ has rank one, $\Theta$ is trivial over the
locus where $\Omega_{\cC/S}$ is locally free.

We now prove that the composite
\begin{equation}\lab{4.5}
\bar\Theta: \omega\dual_{\cC/S}\lra \Omega_{\cC/S}
\end{equation}
of the tautological $\omega_{\cC/S}\dual\to\Omega_{\cC/S}\dual$ with
$\Theta$ is trivial. Because $\Theta$ vanishes at general points, we
only need to show that $\bar\Theta$ is trivial at a node, say $q$,
of the a fiber of $\cC/S$. Since the latter is a local problem, we
can assume that $\cC$ is a family of affine curves. Then by
shrinking $S$ if necessary we can realize $\cC/S$ as a subfamily of
$\tilde\cC/\tilde S$ via $S\sub \tilde S$ and
$\cC=\tilde\cC\times_{\tilde S}S$ so that the node $q$ is smoothed
within the family $\tilde\cC$ and the morphism $f\mh \cC\to X$ is
extended to $\tilde f\mh \tilde \cC\to X$.

For the family $\tilde f$, we form the similar homomorphism
\begin{equation}\lab{4.6}
\tilde\Theta: \omega\dual_{\tilde\cC/\tilde S}\lra \Omega_{\tilde
\cC/\tilde S}.
\end{equation}
Following the definition, its restriction to
$\cC/S\sub\tilde\cC/\tilde S$ is the homomorphism $\bar\Theta$ in
(\ref{4.5}). But for $\tilde\Theta$, it must be trivial since it is
trivial at general points and since $\Omega_{\tilde \cC/\tilde S}$
is torsion free. Therefore, (\ref{4.5}) must be trivial as well.

If we apply $\cE xt^1_{\pi}(-, \cO_{\cC})$ to $\bar\Theta$, we
obtain the vanishing of the composition
\[
\cE xt^1_{\pi}(\Omega_{\cC/S}, \cO_{\cC})\lra R^1\pi_*f^*T_X\to
R^1\pi_*\omega_{\cC/S}\cong \cO_S
\]
which is exactly \eqref{yh2.9} by definition.
\end{proof}

Because the cokernel of $\ext^1_\pi(\Omega_{\cC/S},\cO_{\cC}) \to
R^1\pi_*f^*T_X$ is the restriction to $S$ of the obstruction sheaf
$\Ob_\cM$ of $\mgn(X,\beta)$, the preceding lemma gives us a
canonical cosection $\sigma_S\mh \Ob_M|_S\to \cO_S$. Because this
construction is canonical, these $\sigma_S$ for $S\sub \cM$ descend
to form a sheaf homomorphism
\begin{equation}\lab{4.7}
\sigma:\Ob_\cM\lra\cO_\cM.
\end{equation}

This cosection is surjective away from those stable morphisms that
are $\theta$-null.

\begin{defi} A stable map $u\mh C\to X$ is called \emph{$\theta$-null}
if the composite
$$u\sta(\hat\theta)\circ du: T_{C\reg}\lra u\sta T_X|_{C\reg}\lra
 u\sta\Omega_X|_{C\reg}
$$
is trivial over the regular locus $C\reg$ of $C$.
\end{defi}

\begin{lemm}
The degeneracy locus of the cosection $\sigma$ of $\Ob_\cM$ is the
collection of $\theta$-null stable morphisms in $\cM$.
\end{lemm}

\begin{proof}
By definition, $\sigma$ at $[u\mh C\to X]\in\cM$ is the composition
$$
H^1(C,u^*T_X)\mapright{\hat{\theta}} H^1(C,u^*\Omega_X)\lra
H^1(C,\Omega_C)\lra H^1(C,\omega_C)=\CC
$$
whose Serre dual is
$$\CC=H^0(C,\cO_C)\lra H^0(f^*T_X\otimes \omega_C)\lra
H^0(f^*\Omega_X\otimes \omega_C).
$$
Because $\cO_C$ is generated by global sections, the composite of
the above sequence is trivial if and only if the composite
$$
T_{C}\otimes \omega_C|_{C\reg}\lra f^*T_X\otimes \omega_C|_{C\reg}
\lra f^*\Omega_X\otimes \omega_C|_{C\reg}
$$
is trivial. But this is equivalent to $u$ being $\theta$-null. This
proves the lemma.
\end{proof}

We summarize the above as follows.

\begin{prop}
Any holomorphic two-form $\theta\in H^0(\Omega^2_X)$ on a smooth
quasi-projective variety $X$ induces a cosection $\sigma\mh
\Ob_\cM\lra \cO_{\cM}$ of the obstruction sheaf $\Ob_\cM$ of the
moduli space $\cM$. Its degeneracy locus $Z(\sigma)$ is the set of
all $\theta$-null stable maps in $\cM$.
\end{prop}

\subsection{Localized GW-invariants}

Using the cosection $\sigma$ constructed, we shall construct the
localized GW-invariants of a pair $(X,\theta)$ of a smooth
quasi-projective variety and a two-form $\theta\in H^0(\Omega_X^2)$.

\begin{defi}
We say $\beta$ is $\theta$-proper if for any $g$, the subset of
$\theta$-null stable morphisms in $\mgn(X,\beta)$ is proper.
\end{defi}

Let $\beta\in H_2(X,\ZZ)$ be a $\theta$-proper. For any pair $g$ and
$n$, we continue to denote by $\cM=\mgn(X,\beta)$; we let $\sigma\mh
\Ob_\cM\to\cO_\cM$ be the cosection constructed in the previous
subsection. Because the degeneracy locus $Z(\sigma)$ is the subset
of $\theta$-null stable morphisms, it is proper because $\beta$ is
$\theta$-proper. Therefore, by the result of the previous section,
we have the localized virtual cycle
$$[\cM]\virt\loc\in H\lsta(Z(\sigma),\QQ).
$$
We now define the localized GW-invariants of the $\theta$-proper
class $\beta$ as follows.

We let
$$\ev: \cM\lra X^n
$$
be the evaluation morphism, let $\gamma_1,\cdots,\gamma_n\in
H\sta(X)$, let $\alpha_1,\cdots,\alpha_n\in \ZZ^{\geq 0}$, and let
$\psi_i$ be the first Chern class of the relative cotangent line
bundle of the domain curves at the $i$-th marked point. Because the
localized virtual cycle $[\cM]\virt\loc$ is an ordinary homology
class, via the tautological $H\lsta(Z(\sigma),\QQ)\to
H\lsta(\cM,\QQ)$, we can view it as a homology class in
$H\lsta(\cM,\QQ)$ \footnote{It may happen that $[\cM]\virt\loc\ne0$
while its image in $H\lsta(\cM)$ is trivial. In this case we notice
that the localized GW-invariant just defined vanishes automatically.
Thus viewing $[\cM]\virt\loc$ as a class is $H\lsta(\cM)$ won't
cause any trouble as long as the numerical GW-invariants are
concerned.}.

We define the localized GW-invariant of $X$ with descendants to be
$$\langle\tau_{\alpha_1}(\gamma_1)\cdots\tau_{\alpha_n}(\gamma_n)
\rangle^{X,\mathrm{loc}}_{g,\beta} =\int_{[\cM]\virt\loc}
\ev\sta(\gamma_1\times\cdots\times\gamma_n)\cdot
\psi_1^{\alpha_1}\cdots\psi_n^{\alpha_n}
$$

In case $X$ is proper, every class is $\theta$-proper.

\begin{prop} If $X$ is proper, the localized GW-invariant coincides
with the ordinary GW-invariant of $X$.
\end{prop}

\begin{proof}
This follows directly from that when $\cM$ is proper, the localized
virtual cycle $[\cM]\virt\loc$ coincides with the virtual cycle
$[\cM]\virt$ as homology classes in $H\lsta(\cM)$.
\end{proof}

This immediately gives the following vanishing results of J. Lee and
T. Parker \cite{Lee-Parker, Lee} using deformation of almost complex
structures.

\begin{theo}\lab{vanLP}
Let $X$ be a smooth projective variety endowed with a holomorphic
two-form $\theta$. The virtual fundamental class of the moduli space
$\mgn(X,\beta)$ vanishes unless the class $\beta$ can be represented
by a $\theta$-null stable morphism.

Let $D$ be the closed subset of $p\in X$ such that the
anti-symmetric homomorphism $\hat{\theta}(p):T_{X,p}\to
T\dual_{X,p}$ induced by $\theta$ is not an isomorphism. Then for
any $\beta\ne 0$ and $\gamma_i\in H^*(X)$, the GW-invariant $\langle
\prod \tau_{\alpha_i}(\gamma_i)\rangle^X_{g,\beta}$ vanishes if one
of the classes $\gamma_i$ is Poincar\'e dual to a cycle disjoint
from $D$.
\end{theo}

\begin{proof}
Since $\cM=\mgn(X,\beta)$ is proper, the class $[\cM]\virt$
coincides with the localized virtual class $[\cM]\virt\loc$, which
as shown before is a class supported in the set of $\theta$-null
stable maps in $\cM$. If $\beta$ can not be represented by a
$\theta$-null stable morphism, then the set of $\theta$-null stable
morphisms in $\cM$ is empty. Therefore, $[\cM]\virt\loc=0$, proving
the first statement.

For the second statement, just observe that the image of a
$\theta$-null stable map should be contained in $D$ when $\beta\ne
0$.
\end{proof}

\begin{coro}
Let $(X,\theta)$ be a pair of a smooth projective variety of
dimension $2m$ and a holomorphic two-form $\theta$; let $\theta^m$
be its top wedge. Then the degeneracy locus $D$ of $\theta$ is
$D=zero(\theta^m)$. Suppose $\beta\in H_2(X,\ZZ)$ is not in the
image of $H_2(D,\ZZ)$ by the inclusion $D\hookrightarrow X$, then
the virtual fundamental class $[\cM_{g,n}(X,\beta)]\virt$ is zero.
In particular, when $\theta$ is nondegenerate,
$[\cM_{g,n}(X,\beta)]\virt=0$ for $\beta\ne 0$.
\end{coro}

\subsection{Deformation invariance of the localized GW-invariants}
\lab{subs3.3}

Like the ordinary GW-invariants, the localized GW-invariants are
expected to remain constant under deformation of complex structures.
In the following, we shall prove this for the circumstances relevant
to our study.

We consider a smooth family $\cX/T$ of quasi-projective varieties
over a connected smooth affine curve $T$; we assume that this family
admits a regular homomorphic two-form $\theta\in
\Gamma(\cX,\Omega^2_{\cX/T})$. We let $\beta\in H_2(\cX,\ZZ)$ be a
(fiber) curve class and let
$$\cM=\mgn(\cX/T,\beta)
$$
be the moduli space of stable morphisms to fibers of $\cX/T$ of
fundamental class $\beta$.

Let $f\mh \cC\to \cX$ and $\pi\mh \cC\to \cM$ be the universal
family of this moduli stack; let $\kappa\in H^1(\cX,\cT_{\cX/T})$ be
the Kodaira-Spencer class of the first order deformation of
$\cX/T$--- it is the extension class of the exact sequence of
sheaves of tangent bundles
\[
0\lra \cT_{\cX/T}\lra \cT_\cX\lra \cO_\cX\lra 0.
\]
As shown in \cite{Li-Tian}, the obstruction sheaf $\Ob_{\cM}$ of
$\cM=\mgn(\cX/T,\beta)$ and the relative obstruction sheaf
$\Ob_{\cM/T}$, which is the sheaf whose restriction to each fiber
\[
\cM_t=\cM\times_T t=\mgn(\cX_t,\beta)
\]
is the obstruction sheaf $\Ob_{\cM_t}$ of $\cM_t$, fit into the
exact diagram:
\begin{equation}\lab{6.12}
\begin{CD}\lab{2.12}
\pi\lsta f\sta\cO_\cX @>{f\sta\kappa}>> R^1\pi\lsta f\sta\cT_{\cX/T}
@>{\text{surj}}>> \mathrm{ker}\bl R^1\pi\lsta f\sta\cT_{\cX}\to
R^1\pi\lsta
f\sta\cO_{\cX}\br \\
@| @V{\text{surj}}VV@V{\text{surj}}VV\\
\cO_{\cM} @>>> \Ob_{\cM/T} @>{\text{surj}}>> \Ob_{\cM}\\
\end{CD}
\end{equation}
Applying the previous construction, we check that the form $\theta$
induces a cosection of $R^1\pi\lsta f\sta \cT_{\cX/T}$ that descends
to a cosection
\[
\sigma: \Ob_{\cM/T}\lra \cO_{\cM}.
\]
The restriction of $\sigma$ to each fiber $\cM_t$ is the previously
constructed cosection $\sigma_t$ of $\Ob_{\cM_t}$. We let
$Z(\sigma)$ be the union of $Z(\sigma_t)\sub\cM_t$ for all $t\in T$.

Suppose $Z(\sigma)$ is proper over $T$, for each $t\in T$ we can
define the localized GW-invariants of $\cX_t$:
$$\langle\tau_{\alpha_1}(\gamma_1)\cdots\tau_{\alpha_n}(\gamma_n)
\rangle^{\cX_t,\mathrm{loc}}_{g,\beta},\quad \gamma_i\in
H\sta(\cX,\ZZ).
$$
The deformation invariance principle states that the above is
independent of $t$.

In this section, we shall prove this principle for the localized
GW-invariants for the circumstances relevant to our study.

According to Proposition \ref{constancy}, the constancy of the
localized GW-invariants follows from the lifting of the homomorphism
$\sigma$ to a homomorphism
$$\sigma' : \Ob_\cM\lra\cO_\cM,
$$
which by the lower exact sequence in (\ref{6.12}) amounts to the
vanishing of the composite
\begin{equation}\lab{comp-3}
\cO_\cM\lra \Ob_{\cM/T}\mapright{\sigma}\cO_\cM.
\end{equation}

We will prove the vanishing of this composite under certain
additional technical conditions.

The first case we will investigate concerns a smooth family of
projective varieties $\cX/T$ over a connected smooth affine curve
$T$.

\begin{lemm}\lab{con} Under the assumption, the composite
$$\cO_{\cM}=\pi\lsta f\sta\cO_{\cX}\mapright{f\sta\kappa}
R^1\pi\lsta f\sta\cT_{\cX/T}\lra\Ob_{\cM/T}\mapright{\sigma}\cO_\cM
$$
is a lift of a section of $\cO_T$ via the pullback
$p\sta\cO_T\to\cO_\cM$ for $p\mh \cM\to T$.
\end{lemm}

\begin{proof}
Let $\alpha\in\Gamma(\cO_\cM)$ be this section. Since $T$ is
reduced, we only need to check that to each closed $t\in T$ the
restriction $\alpha_t=\alpha|_{\cM_t}$ is a constant section of
$\cO_{\cM_t}$.

For this, we contract the Kodaira-Spencer class $\kappa_t\in
H^1(\cX_t,\cT_{\cX_t})$ of the family $\cX/T$ with the holomorphic
two-form $\theta_t\in \Gamma(\cX_t,\Omega_{\cX_t}^2)$ to obtain a
class $\theta_t(\kappa_t)\in H^1(\cX_t,\Omega_{\cX_t})$. Clearly, if
we represent $\kappa_t$ by a $\cT_{\cX_t}$-valued anti-holomorphic
$1$-form, $\theta_t(\kappa_t)$ becomes a $\overline\partial$-closed
$(1,1)$-form, thus representing a Dolbeault cohomology class. Then
since $\cX_t$ is K\"ahler, by Hodge theory there is a $d$-closed
$(1,1)$-form $u$ and a $(1,0)$-form $v$ so that
$\theta_t(\kappa_t)=u+\overline\partial v$. But then, since the
section $\alpha_t\in\Gamma(\cO_{\cM_t})$ is defined via the
push-forward $\pi\lsta f\sta\theta_t(\kappa_t)$, it is equal to
$\pi\lsta f\sta u+\pi\lsta f\sta \overline{\partial}v=\pi\lsta f\sta
u$, because $\pi\lsta f\sta \overline{\partial}v=0$; since $u$ is a
cohomology class in $H^2(\cX_t,\CC)$, $\alpha_t$ is identical to the
constant $\int_\beta u\in\CC$. This proves the lemma.
\end{proof}

\begin{coro}\lab{trivial}
Suppose further that the class $\beta$ can be represented by a
$\theta_t$-null stable map. Then the pairing $\int_\beta u$
vanishes; so does the section $\alpha_t$; and so does the composite
(\ref{comp-3}).
\end{coro}

The second case is the deformation to the normal cone of a smooth
canonical curve in a surface $X$. Let $(X,\theta)$ be a pair of a
smooth surface and a holomorphic two-form with smooth degeneracy
locus $D=\theta\upmo(0)$. By blowing up $X\times \Ao$ along $D\times
0$ we obtain a family of proper surfaces $\pi\mh Z\to\Ao$. We then
let $\tilde Z\sub Z$ be the complement of the proper transform of
$X\times 0$ in $Z$. The restricted family $\tilde Z\to\Ao$ is the
union of $X\times(\Ao-0)$ with the total space of the normal bundle
$N_{D/X}$. In this case, since the sheaf of the tangent bundle
$\cT_Z$ restricted to the exceptional divisor $E\sub Z$ is
isomorphic to $\cT_E\oplus \cO_E(-1)$, the Kodaira-Spencer class
$\kappa_0$ of the family $\tilde Z/\Ao$ along the central fiber is
trivial. Thus the conclusion of Corollary \ref{trivial} holds for
this family. We remark that by adjunction formula, $N_{D/X}^{\otimes
2}\cong K_D$. Thus $\tilde Z_0$ is the total space of a theta
characteristic of $D$.

The last case we shall consider concerns the total space of theta
characteristics of smooth curves. Recall that a theta characteristic
of a smooth curve $D$ is a line bundle $L$ so that $L^{\otimes
2}\cong K_D$. Give such a pair $(D,L)$, the total space $X$ of the
line bundle $L$ is a surface whose canonical line bundle $K_X$ is
the pullback of $L$ under the tautological projection $\pi\mh X\to
D$. Thus the assignment that sends any $x\in X$ to the same $x\in
L_{\pi(x)}$ defines a canonical section $\theta\in\Gamma(X,K_X)$;
the vanishing locus of $\theta$ is exactly the zero section of
$\pi$. We call this two-form the \emph{standard holomorphic
two-form} on $X$.

In the following, we let $D_t$ be a smooth family of curves and
$L_t$ be a family of theta characteristics of $D_t$. The total
spaces $\cX_t$ of $L_t$ form a smooth family $\cX/T$ with $\theta\in
\Gamma(\cX,\Omega_{\cX/T}^2)$ the standard relative holomorphic
two-form. We let $\beta\in H_2(\cX,\ZZ)$ be the class generated by
the zero section of one of $\cX_t$. As before, we denote
$\mgn(\cX/T,d\beta)$ by $\cM$.

\begin{lemm} The conclusion of Lemma \ref{con} holds for
the family $\cX$ and the form $\theta$.
\end{lemm}

\begin{proof}
To prove the lemma, we first check that for the Kodaira-Spencer
class $\kappa_0\in H^1({\cX_0},T_{\cX_0})$ of the family $\cX/T$ and
$(f,\cC)$ the universal family of $\cM_0=\mgn({\cX_0},d\beta)$, we
have that the composite
\begin{equation}\lab{comp}
\cO_{\cM_0}=\pi\lsta f\sta\cO_{\cX_0}\lra R^1\pi\lsta
f\sta\cT_{\cX_0}\lra \Ob_{\cM_0}\mapright{\sigma}\cO_{\cM_0}
\end{equation}
is locally constant.

We first describe the Kodaira-Spencer class $\kappa_0$, which
depends on the family $D_t$. For simplicity, we shall work with
analytic charts of $D_0$. We pick an analytic open $U\sub D_0$ so
that $U$ is isomorphic to the unit disk $\Delta\sub\CC$. We then let
$V=D_0-A$ with $A\sub U$ a compact subset so that $U-A$ is
isomorphic to an annulus. The two open sets $U$ and $V$ form an open
covering of $D_0$. Since $H^1(T_U)=H^1(T_V)=0$, $H^0(T_{U\cap V})\to
H^1(T_{D_0})$ is surjective. Hence, for small $t$ the family $D_t$
can be realized by an analytic deformation of the gluing map
$$U\supset U\cap V\mapright{\cong} U\cap V\sub V.
$$
In concrete terms, if we let $z$ and $w$ be the analytic coordinates
of $U$ and $V$ near $U\cap V$, and let $z=f(w,0)$ be the identity
map of $U\cap V$ in coordinate variables $z$ and $w$, then $D_t$ can
be realized by gluing $U$ and $V$ via $z=f(w,t)$ with $f(w,t)$ an
analytic deformation of $f(w,0)$.

To proceed, we need the transition function of $\cX_t$. Because
$D_t=U\cup V$, the surface $\cX_t$ is the union of the total space
of $K_U^{\frac{1}{2}}$ and $K_V^{\frac{1}{2}}$. As to the transition
function of $\cX_t$, we let $\xi=(dz)^{\frac{1}{2}}$ and
$\eta=(dw)^{\frac{1}{2}}$ be bases of $K_U^{\frac{1}{2}}$ and
$K_V^{\frac{1}{2}}$ over ${U\cap V}$. Then by adopting the
convention that $f_w=\frac{\partial f}{\partial w}$ and $\dot
f=\frac{df}{dt}$, the two pairs of local charts $(z,\xi)$ and
$(w,\eta)$ are related by
$$z=f(w,t)\and \xi=(f_w)^{\frac{1}{2}}\eta.
$$
Accordingly, the Kodaira-Spencer class of the first order
deformation of $\cX_t$ at $t=0$ can be represented by \v Cech
1-cocycle
$$\kappa_0(U\cap V)=\bl\frac{d}{dt}(f)\cdot\frac{\partial}{\partial z}+
\frac{d}{dt}((f_w)^{\frac{1}{2}}\eta)\cdot \frac{\partial}{\partial
\xi}\br|_{t=0}= \bl \dot f\cdot\frac{\partial}{\partial
z}+\frac{\xi\,\dot f_w}{2f_w}\cdot \frac{\partial}{\partial
\xi}\br|_{t=0}.
$$
Because over $K_U^{\frac{1}{2}}\cap K_V^{\frac{1}{2}}$, the standard
holomorphic two-form is $\theta_0=\xi\, d\xi\wedge dz$, the
contraction is
$$\theta_0(\kappa_0)=-\xi\dot f \,d\xi+\frac{1}{2}\xi^2{\dot
f_w}{f_w\upmo} dz.
$$
Therefore, using $\dot f_z=\dot f_w\frac{\partial w}{\partial
z}=-\dot f_w f_w\upmo$,
$$\partial(\theta(\kappa_0))=-\xi\dot f_z dz\wedge d\xi+\xi\dot f_w
f_w\upmo d \xi\wedge dz=0.
$$
Combined with the fact that $\bar\partial \theta(\kappa_0)=0$, we
see that the form $\theta_0(\kappa_0)$ is $d$-closed.

The lemma now follows easily. We let $p\mh {\cX_0}\to D_0$ be the
projection and let $\cN\sub \cM_0$ be the (analytic) open subset
consisting of those $h\mh C\to {\cX_0}$ so that $p\circ h\mh C\to
D_0$ are unramified over $U\cap V$. We then pick an oriented
embedded circle $S^1\sub U\cap V$ that separates the two boundary
components of $U\cap V$. An easy argument shows that the
homomorphism (\ref{comp}) is the function, up to sign,
$$(f,C)\in \cM_0\longmapsto \int_{f\upmo(S^1)}\theta_0(\kappa_0)\in \CC.
$$
Because $\theta_0(\kappa_0)$ is $d$-closed, this integral only
depends on the topological class of $f\upmo(S^1)$, hence must be
locally constant over $\cN$. But then this constant must be zero
since it vanishes on those $h$ so that $h(C)\sub D_0\sub \cX_0$, and
since by dilation along fibers of $L_0$ each $h\mh C\to {\cX_0}$ can
be deformed to a stable map from $C$ to $D_0\sub {\cX_0}$ within
$\cN$. This shows that (\ref{comp}) is zero over $\cN$.

To complete the proof, we observe that for each stable map in
$\cM_0$, we can choose $U\sub D_0$ so that this stable map lies in
the $\cN$ associated to $U$. Therefore (\ref{comp}) must be zero on
all $\cM_0$, completing the proof of the lemma.
\end{proof}

\subsection{Reduction for surfaces}

Let $S$ be a smooth {general type minimal surface} equipped with a
holomorphic two-form $\theta$ whose vanishing locus is a canonical
curve $B\in |K_S|$. In case $B$ is smooth, {$B$ is connected} and by
the adjunction formula, $K_B=K_S^{\otimes 2}|_B$, the genus of $B$
is $g(B)=K_S^2+1$ and the normal bundle $N_{B/S}\cong K_S|_B$ is a
theta characteristic of $B$, namely $N_{B/S}^{\otimes 2}\cong K_B$,
of parity (\cite{Lee-Parker, MaulikPand})
$$h^0(B,N_{B/S})\equiv \chi(\cO_S)\mod 2.
$$

Conversely, for a smooth projective curve $D$ of genus $K_S^2+1$ and
a theta characteristic $L$ of $D$ of parity $\chi(\cO_S) \mod 2$, we
take $p\mh X\to D$ to be the total space of the line bundle $L$ and
let $\theta\in H^0(X,\Omega_X^2)$ be the standard two-form on $X$.
Since $\theta\upmo(0)=D$ is proper, the localized GW-invariants of
$X$ are well-defined.

{The GW-invariants of $S$ are expected to be determined by the
numerical date $K_S^2$ and $\chi(\cO_S)$.}

\begin{conj}\lab{reduct}
Let $S$ be a smooth {minimal general type surface with positive
$p_g>0$}. Its GW-invariants $\langle \cdots \rangle _{\beta,g}^S$
vanish unless $\beta$ is a non-negative integral multiple of
$c_1(K_S)$. In case $\beta=d c_1(K_S)$ for an integer $d> 0$, we let
$(D,L)$ be a pair of a smooth projective curve of genus $K_S^2+1$
and its theta characteristic with parity $\chi(\cO_S)$, and let $X$
be the total space of $L$. Then there is a canonical homomorphism
$\rho\mh H^{\ast}(S,\ZZ)\to H^{\ast}(X,\ZZ)$ so that for any classes
$\gamma_i\in H^{\ast}(S,\ZZ)$ and integers $\alpha_i\geq 0$,
$i=1,\cdots,n$,
\[
\langle\tau_{\alpha_1}(\gamma_1)\cdots\tau_{\alpha_n}(\gamma_n)\rangle^S_{\beta,g}
=\langle\tau_{\alpha_1}(\rho(\gamma_1))\cdots
\tau_{\alpha_n}(\rho(\gamma_n))\rangle^{X,\mathrm{loc}}_{d[D],g}.
\]
\end{conj}

In case $S$ has a smooth canonical divisor $D\in |K_S|$, this is
proved by Lee-Parker using symplectic geometry in \cite{Lee-Parker}.
In \S \ref{subs3.3} above, we proved this algebraically by showing
invariance of the localized virtual fundamental class under
deformation of $S$ to the normal bundle $X$ of $D$. More generally,
by constructing a deformation of complex structures of an analytic
neighborhood $U$ of $D\in |K_S|$ in $S$, we can confirm this
conjecture for a wider class of surfaces. We shall address this in
our forthcoming paper \cite{KL3}.

In light of this {conjecture}, in the remainder of this paper, we
shall concentrate on studying the localized GW-invariants of a
surface $p\mh X\to D$ that is the total space of a theta
characteristic $L$ of a smooth projective curve $D$ together with
its standard holomorphic two-form $\theta$ on $X$.

\subsection{Relation with the twisted invariants}

The localized GW-invariant
$\langle\cdot\rangle^{X,\mathrm{loc}}_{d[D],g}$ is expected to
relate to the twisted GW-invariants of the curve $D$.

To begin with, the projection $p\mh X\to D$ induces a morphism from
the moduli of stable morphisms to $X$ to the moduli of stable
morphisms to $D$
$$\bar p: \mgn(X,d)\lra \mgn(D,d),\quad d>0.
$$
Here since $H_2(X,\ZZ)$ is generated by the zero section $D\sub X$,
we will use integer $d$ to stand for the class $d[D]$. At individual
map level, it is clear that in case $h\mh C\to X$ is a stable map,
then $p\circ h\mh C\to D$ is a stable map and the original $h$ is
determined by a global section $H^0(C,(p\circ h)\sta L)$. Thus
fibers of $\bar p$ are vector spaces.

We next put this into family. For convenience, in the following, we
shall fix $g,n$ and $d$ and abbreviate $\mgn(X,d)$ and $\mgn(D,d)$
to $\cMM$ and $\cNN$, respectively. We let
\[\tilde f: \tilde\cC\lra X \and \tilde\pi:
\tilde \cC\lra \cMM
\]
be the universal family of $\cMM$ and let
\begin{equation}\lab{yh5.6} f:
\cC\lra D \and \pi:  \cC\lra \cNN
\end{equation}
be the universal family of $\cNN$. Because $X\to D$ has affine
fibers, the composite $p\circ \tilde f: \tilde\cC\to D$ is a family
of stable morphisms as well. Therefore,
$\tilde\cC=\cC\times_{\cNN}\cMM$ and $p\circ \tilde f$ is the
pullback of $f$.

We next pick two vector bundles $E_1$ and $E_2$ on $\cNN$ and a
sheaf homomorphism
$$\alpha:\cO_{\cNN}(E_1)\lra \cO_{\cNN}(E_2)
$$
whose cohomology gives us the complex $\pi_!f\sta L$. By viewing
$q\mh E_1\to\cNN$ as the total space of $E_1$ with $q$ the
projection, we can form the pullback bundle $q\sta E_2$ on $E_1$ and
the associated section $\bar\alpha\in \Gamma(E_1, q\sta E_2)$.

\begin{lemm} \lab{yh5.5}
The vanishing locus $\bar\alpha\upmo(0)$ is the moduli stack $\cMM$.
\end{lemm}

\begin{proof}
The proof is straightforward and will be omitted.
\end{proof}

It was hoped that the localized GW-invariants of $X$ can be
recovered by the twisted GW-invariants of $D$. As will be shown
later, this unfortunately fails in general. However, in case the
sheaf $R^1\pi\lsta f\sta L$ is locally free, this is true up to
sign.

In the remainder of this section, we assume that $R^1\pi\lsta f\sta
L$ is locally free. Accordingly, we can take
$E_1=\pi\lsta f\sta L$, $E_2=R^{1}\pi\lsta f\sta L$ and $\alpha=0$.
Thus $\cMM=\bar\alpha\upmo(0)$ is the total space of $E_1$ and
$q=\bar p$.

In the following, we shall investigate the cosection $\sigma$ more
closely.

\begin{lemm}
Let the notation be as above and suppose $d>0$. Then there are two
canonical homomorphisms $\nu$ and $\bar\theta\dual$ as shown
\begin{equation}\lab{yh5.3-1}
\Ob_{\cMM}\mapright{\nu} R^1\tilde{\pi}_*(\tilde{f}^*p\sta
L^\vee\otimes
\omega_{\tilde{\cC}/\cMM})\mapright{\bar\theta\dual}\cO_{\cMM}
\end{equation}
so that their composite is the associated cosection $\sigma$ of
$\Ob_{\cMM}$. Further, $\nu$ is surjective and the middle term above
is isomorphic to $\bar p\sta E_1\dual$.
\end{lemm}

\begin{proof}
Because the two-form $\theta$ on $X$ is a section of
$p^*L=\Omega^2_X$, it provides a section and its dual:
\begin{equation}\lab{yh5.1}
\bar \theta \in H^0\bl\cMM, \tilde{\pi}_*\tilde f\sta p\sta L\br
\and \bar\theta^\vee \in \Hom\bl R^1\tilde{\pi}_*(\tilde f\sta p\sta
L^\vee\otimes \omega_{\tilde{\cC}/\cMM}), \cO_{\cMM}\br.
\end{equation}
Note that although $\tilde{\pi}_*\tilde f\sta p\sta L=\bar p\sta
{\pi}_* f\sta L$, 
$\bar\theta$ is not a pullback from a section over $\cNN$ since it
is injective on $\cMM-\cNN$ while vanishing on $\cNN$. Since $L\cong
L^\vee\otimes K_D$, we further have a morphism
\[
\tilde f\sta p\sta L\cong \tilde f\sta p\sta L^\vee\otimes \tilde
f\sta p\sta K_D \lra \tilde f\sta p\sta L^\vee \otimes
\omega_{\tilde\cC/\cMM}
\]
that induces a homomorphism of vector bundles \begeq\lab{yh5.2}
R^1\tilde{\pi}_*\tilde f\sta p\sta L\lra R^1\tilde{\pi}_*(\tilde
f\sta p\sta L^\vee\otimes \omega_{\tilde\cC/\cMM});
\endeq
this homomorphism is surjective because $H^0_C(\tilde f\sta p\sta
L)\to H^0_C(\tilde f\sta p\sta L^\vee\otimes \omega_C)$, which is
the Serre dual of $H^1_C(\tilde f\sta p\sta L)\to H^1_C(\tilde f\sta
p\sta L^\vee\otimes \omega_C)$, is injective for any stable map
$C\to X$ of degree $d>0$.
%
%

Obviously, this homomorphism is the pullback of the corresponding
\emph{surjective} homomorphism \begeq\lab{yh5.21} E_2=R^1{\pi}_*
f\sta L\lra R^1{\pi}_*(f\sta L^\vee\otimes \omega_{\cC/\cNN})
=E_1\dual
\endeq
over $\cNN$ via $\bar p$. We let $V$ be the kernel vector bundle of
this homomorphism.

We claim that the cosection $\sigma :Ob_{\cMM}\lra \cO_{\cMM}$
defined in \eqref{4.7} factors through the homomorphism
$\bar\theta\dual$ in \eqref{yh5.1}. Indeed, from the natural
morphisms
\[
\tilde{f} ^*(T_X\otimes p^*L)=\tilde{f}^*(T_X\otimes \Omega^2_X)\lra
\tilde{f}^*\Omega_X\lra \Omega_{\tilde{\cC}/\cMM}\lra
\omega_{\tilde\cC/\cMM}
\]
we obtain a homomorphism $\tilde{f}^*T_X\to \tilde f\sta p\sta
L^\vee\otimes \omega_{\tilde\cC/\cMM}$ and thus a homomorphism
\begin{equation}\lab{12}
R^1\tilde{\pi}_*\tilde{f}^*T_X\lra R^1\tilde{\pi}_*(\tilde f\sta
p\sta L^\vee\otimes \omega_{\tilde\cC/\cMM}).
\end{equation}
Similar to the proof of Lemma \ref{lem-yh2.3}, its composition with
the natural homomorphism
\[
\cE xt^1_{\tilde{\pi}}(\Omega_{\tilde{\cC}/\cMM},
\cO_{\tilde{\cC}})\lra R^1\tilde{\pi}_*\tilde{f}^*T_X
\]
is zero. Thus \eqref{12} lifts to a homomorphism
\begin{equation}\lab{yh5.3}
\nu: Ob_{\cMM}\lra R^1\tilde{\pi}_*(\tilde f\sta p\sta L^\vee\otimes
\omega_{\tilde\cC/\cMM}),
\end{equation}
which by construction satisfies $\bar\theta^\vee\circ \nu=\sigma$,
as desired.


Finally, because $R^1\tilde{\pi}_*  \tilde f\sta p\sta L\to
\Ob_{\cMM}$ composed with \eqref{yh5.3} is \eqref{yh5.2} that is
surjective, the arrow $\nu$ is surjective as well. The last
isomorphism in the statement of the lemma follows from the Serre
duality.
\end{proof}

As to the arrow $\bar\theta\dual$, since $R^1\tilde\pi\lsta\bl\tilde
f\sta p\sta L\dual\otimes\omega_{\tilde\cC/\cMM}\br$ is dual to
$\tilde\pi\lsta \tilde f\sta p\sta L=\bar p\sta\pi\lsta f\sta L=\bar
p\sta E_1$ and since $\cMM$ is the total space of $E_1$, an obvious
choice of a cosection like $\bar\theta\dual$ is via the standard
pairing $E_1\dual\times E_1\to\CC$. A direct check shows that this
is the case.

We now state and prove the main result of this section that relates
the localized GW-invariants to the twisted GW-invariants of $D$.

\begin{prop}\lab{fundtwist}
Let the notation be as before and suppose $R^1\pi_*f^*L$ is a
locally free sheaf on $\cNN$. Suppose $d>0$. Then the localized
virtual cycle satisfies
\[
[\cMM]\loc\virt=\sum_{\cNN_i\sub\cNN}(-1)^{h^0(f^*L)}
[\cNN_i]{\virt}\cap e(V),
\]
where $V$ is the kernel vector bundle of \eqref{yh5.21}; the
summation is over all connected components $\cNN_i$ of $\cNN$. Here
$h^0(f^*L)$ denotes the rank of $\pi_*f^*L$ over each component and
$e(V)$ is the Euler class of $V$.
\end{prop}

\begin{proof}
Let $\cO_\cNN(\tilde F)\to\Ob_{\cMM}$ be an epimorphism of a locally
free sheaf on $\cMM$ to the obstruction sheaf of $\cMM$.  The
obstruction theory of $\cMM$ provides us with a virtual normal cone
$C_{\cMM}\sub \tilde F$. Since the restriction $\Ob_{\cMM}|_{\cNN}$
splits as the direct sum $\Ob_\cNN\oplus E_2$, $\tilde F|_\cNN$ is
an extension of $E_2$ by $F=\mathrm{Ker}(\tilde
F|_\cNN\twoheadrightarrow E_2)$, and $F$ surjects onto $\Ob_\cNN$.

Because $\nu$ is surjective, by composing $\nu$ with the
homomorphism $\cO_{\cMM}(\tilde F)\to\Ob_{\cMM}$ we obtain a
surjective homomorphism of vector bundles $\tilde F\to \bar p\sta
E_1\dual$. We let $W\sub \tilde F$ be its kernel. By the proof of
Lemma \ref{j2.1}, the support of the cone $C_{\cMM}$ lies in the
subvector bundle $W\sub\tilde F$.

We now look at the localized virtual cycle $[\cMM]\virt\loc$. First,
by our construction it is the image of $[C_{\cMM}]$ under the
localized Gysin map
$$[\cMM]\virt\loc=s_{\tilde F,\mathrm{loc}}^![C_{\cMM}]\in H\lsta(\cNN,\QQ).
$$
(See notation in section 2.) Since $C_{\cMM}$ lies in the kernel of
$\tilde F\to \bar p\sta E_1\dual$ while the cosection $\tilde
F\to\Ob_{\cMM}\to\cO_{\cMM}$ factors through $\tilde F\to \bar p\sta
E_1\dual$, we can deform the bundle $\tilde F$ to $W\oplus \bar
p\sta E_1\dual$ so that $C_{\cMM}\sub W$ and $\bar p\sta
E_1\dual\to\cO_{\cMM}$ remain unchanged. Then we must have
$$s_{\tilde F,\mathrm{loc}}^![C_{\cMM}]=
s_{W\oplus \bar p\sta E_1\dual,\mathrm{loc}} ^![C_{\cMM}]
=s_W^![C_{\cMM}\times_{\cMM} s^!_{\bar p\sta E_1\dual,\mathrm{loc}}
[0_{\bar p\sta E_1\dual}]].
$$
Now suppose $\cNN$ is connected and $r=\rank E_1$. An easy
computation shows that
$$s^!_{\bar p\sta E_1\dual,\mathrm{loc}} [0_{\bar p\sta E_1\dual}]=(-1)^r[\cNN],
$$
where $[\cNN]$ is the fundamental class of $\cNN$. Then since over
$\cNN\sub\cMM$, $W=V\oplus F$ and $C_{\cMM}\times_{\cMM}\cNN$
coincides with the virtual normal cone $C_\cNN\sub F$ of $\cNN$, we
have
$$s_W^![C_{\cMM}\times_{\cMM} s^!_{\bar p\sta
E_1\dual,\mathrm{loc}} [0_{\bar p\sta E_1\dual}]]=(-1)^r e(V)\cap
s^!_{F}[C_\cNN]=(-1)^r e(V)\cap [\cNN]\virt.
$$
This proves the proposition in case $\cNN$ is connected.

The general case can be treated in exactly the same way by working
over each individual connected component of $\cNN$.
\end{proof}

\begin{coro}\lab{twist}
Suppose furthermore that $\pi_*f^*L$ is a trivial vector bundle on
$\cNN$, then
\[
[\cMM]\virt\loc=(-1)^{h^0(f^*L)} [\cNN]{\virt}\cap c_l(-\pi_!f^*L)
\]
for $l$ the rank of $-\pi_!f^*L=R^1\pi_*f^*L - \pi_*f^*L$.
\end{coro}
\begin{proof} Observe that $e(V)=c_l(-\pi_!f^*L)$ under the
assumptions.
\end{proof}

\subsection{The case of \' etale coverings}

The above formula enables us to recover the GW-invariants of a
surface $S$ without descendant insertions, first proved by
Lee-Parker \cite{Lee-Parker}.

As before, let $S$ be a smooth {minimal general type} surface with
$p_g\ne 0$ and let $\theta\in H^0(\Omega_S^2)$ be a general member;
let $(\beta,g)\in H_2(S,\ZZ)\times \ZZ^{\ge 0}$. By Theorem
\ref{reduct}, the GW-invariants vanish unless $\beta=dK_S$. In case
$\beta=dK_S$, for $\gamma_i\in H^{*}(S,\QQ)$, \begeq\lab{GW-1}
\langle \gamma_1,\cdots,\gamma_n \rangle^S_{\beta,g}=0
\endeq
if one of $\gamma_i\in H^{\ge 3}(S,\QQ)$, by Theorem \ref{vanLP}
(2). On the other hand, if all $\gamma_i\in H^{\le 2}(S,\QQ)$ for
all $i$ and $\gamma_i\in H^{\le 1}(S,\QQ)$ for at least one $i$,
then \eqref{GW-1} holds by dimension count. Hence we are reduced to
consider the case $\gamma_i\in H^2(S,\QQ)$. But then {according to
Conjecture} \ref{reduct}, we {expect}
\begin{equation}
\lab{yh6.6} \langle
\gamma_1,\cdots,\gamma_n\rangle^S_{dK_S,g}=d^n\prod_{i=1}^n(\gamma_i\cdot
K_S) \cdot\langle 1\rangle^S_{dK_S,g}=d^n\prod_{i=1}^n(\gamma_i\cdot
K_S) \cdot\langle 1\rangle^{X,\mathrm{loc}}_{d,g}
\end{equation}
where $X$ is the total space of a theta characteristic of parity
$\chi(\cO_S)$ on a smooth curve $D$ of genus $h=K_S^2+1$. As
mentioned above, this equality holds true for a wide class of
minimal surfaces of general type with $p_g>0$, including those which
admit smooth canonical curves. Obviously, {the right hand side} of
\eqref{yh6.6} is trivial unless the virtual dimension of
$\cM_{g,0}(X,d)$ is zero, which happens exactly when
$g-1=dK_S^2=d(h-1)$.

{We now consider the only non-trivial case $g=d(h-1)+1$. This choice
of $g$ forces any $u\in \cM_{g,0}(D,d)$ to be a $d$-fold \'etale
cover of $D$. Since an \'etale cover $u:C\to D$ is a smooth and
isolated point in $\cM_{g,0}(D,d)$, a direct application of Theorem
\ref{reduct} and Corollary \ref{twist} shows}

{
\begin{prop}[\cite{Lee-Parker}]
\[
\langle 1\rangle ^{X,\mathrm{loc}}_{d,g} = \sum_{u:d\text{-fold
\'etale cover of }D} \frac{(-1)^{h^0(u^*L)}}{| \mathrm{Aut}(u) |}.
\]
\end{prop}
}

\section{Low-degree GW-invariants of surfaces}

In the remainder of this paper, we shall use degeneration to
determine low degree GW-invariants of surfaces with positive $p_g$.
To keep this paper in a manageable size, we will only state the
degeneration formula for localized GW-invariants and prove the
closed formulas for degree one and degree two localized
GW-invariants. These formulas were conjectured by Maulik and
Pandharipande \cite{MaulikPand}.

\subsection{A degeneration formula}\lab{section4.1}

As degeneration formulas are more easily expressible for the
GW-invariants of stable morphisms with not necessary connected
domains, we shall work with such moduli spaces in the remainder of
this paper. As such, for integers $n, \chi$ and homology class
$\beta\in H_2(Y,\ZZ)$ of a smooth projective variety $Y$, we shall
denote by
$$\mxn(Y,\beta)^\bullet
$$
the moduli of stable morphisms $f\mh C\to Y$ to $Y$ from not
necessarily connected $n$-pointed nodal curves $C$ of Euler
characteristic $\chi(\cO_C)=\chi$ and of fundamental class
$f\lsta([C])=\beta$, such that the restriction of $f$ to each
connected component of $C$ is non-constant. As usual, we call $f$
stable when the automorphism group of $f$ is finite.

Now suppose $Y$ has a holomorphic two-form $\theta$. The preceding
discussion of localized GW-invariants can be adopted line by line to
this moduli space. Consequently, we shall denote the resulting
GW-invariants, of stable maps with not necessarily connected domain
but with non-constant restriction to any connected components, by
$$
\langle  \prod_{j=1}^{n}\tau_{\alpha_j}(\gamma_j)
\rangle^{Y,\bullet}_{\beta,\chi,\mathrm{loc}}.
$$

For relative invariants, we fix a pair $(Y,E)$ with $E\sub Y$ a
smooth divisor; we fix the data $(\chi,\beta,n)$ plus a partition
$\eta$ of $\int_E\beta$; namely,
$\eta=(0\leq\eta_1\leq\cdots\leq\eta_\ell)$ with
$\sum\eta_i=\int_E\beta$. Following the convention, we shall write
$\ell=\ell(\eta)$ and $|\eta|=\sum\eta_i$. We then form the moduli
of relative stable morphisms that consists of morphisms
$$f\mh (C,p_1,\cdots,p_n, q_1,\cdots,q_\ell)\to Y
$$
from $(n+\ell)$-pointed curves to $Y$ as before with one additional
requirement: as divisors
$$f\upmo(E)=\sum \eta_i q_i.
$$
For such $f$, we define its automorphisms be $\varphi\mh
C\mapright{\cong} C$ so that $\varphi$ fixes $p_i$ and $q_j$, and
$f\circ\varphi=f$. We call $f$ stable if the automorphism group of
$f$ is finite.

This moduli space is a Deligne-Mumford stack; it is not proper in
general. In \cite{Li1}, the second author constructed its
compactification by considering all relative stable maps to the
semi-stable models of $(Y,E)$. The resulting moduli space
$\mxn^{(Y,E)}(\beta,\eta)\bul$ is proper, has perfect obstruction
theory, and has two evaluation morphisms: one ordinary and one
special:
$$ev(f)=\bl f(p_1),\cdots, f(p_n)\br\in Y^n
,\quad \tilde{ev}(f)=(f(q_1),\cdots,f(q_\ell))\in E^{\ell(\eta)}.
$$
For classes $\gamma_i\in H\sta(Y)$ and non-negative integers
$\alpha_i$, the relative GW-invariants
$$
\langle\tau_{\alpha_1}(\gamma_1)\cdots
\tau_{\alpha_n}(\gamma_n)\rangle^{(Y,E),\bullet}_{\chi,\beta,\eta}
$$
is the direct image
$$
\tilde{ev}\lsta\left( \int_{[\mxn^{(Y,E)}(\beta,\eta)^\bullet]\virt}
ev\sta(\gamma_i\times\cdots\times\gamma_n)\cdot
\psi_1^{\alpha_1}\cdots\psi_n^{\alpha_n}\right)\in
H\lsta(E^{\ell(\eta)}).
$$

The relative invariants allow us to use degeneration to study the
GW-invariants of smooth varieties. Suppose a smooth variety $W$
specialize to a union of two smooth varieties $Y_1\cup Y_2$
intersecting transversally at $E=Y_1\cap Y_2$, the GW-invariants of
$W$ can be recovered by the relative GW-invariants of the pairs
$(Y_1,E)$ and $(Y_2,E)$.

In this section, we shall state a parallel degeneration theory for
the localized GW-invariants of $X$ that is the total space of a
theta characteristic over a smooth curve $D$.

To this end, we continue denoting by $X$  the total space of a theta
characteristic $L$ of a smooth curve $D$; we pick a point $q\in D$
and denote by $E$ the fiber of $X$ over $q$. We then blow up
$X\times\Ao$ along $E\times 0$, resulting a family $\cX$ over $\Ao$
whose fiber over $t\ne 0$ is the original $X$; its central fiber
$\cX_0$ is the union of $X$ with $E\times\Po$, intersecting
transversally along $E\sub X$ and $E\times 0\sub E\times\Po$. To
distinguish the $X\sub \cX_0$ from $\cX_t$, we shall denote by
$Y_1\sub \cX_0$ the component $X$ and denote $E\times\Po\sub \cX_0$
by $Y_2$. We continuous to denote $E=Y_1\cap Y_2$. Since $X$ is the
total space of the line bundle $L$ over $D$, $\cX$ is the total
space of a line bundle $\cL$ over $\cD$ of which $\cD$ is the
blowing up of $D\times\Ao$ along $(q,0)$ and $\cL$ is the pull back
of $L$ via the composite of the projections $\cX\to \cD\to D$. We
denote the projection by $p\mh \cX\to\cD$ and denote $D_i=p(Y_i)\sub
\cD_0$.

Our next step is to extend the standard holomorphic two-form
$\theta$ of $X$ to the family $\cX$. First, since $\cX$ is the total
space of $\cL$ over $\cD$, $\cL$ admits a tautological section
$id\in\Gamma(\cX,p\sta\cL)$ that takes value $v\in \cL$ at the point
$v\in\cX$. On the other hand, the isomorphism $L^{\otimes 2}\cong
K_D$ induces an isomorphism $\cL^{\otimes 2}\to
\omega_{\cD/T}(-D_2)$. In this way, the tautological section $id$
composed with $p\sta\cL\to p\sta\cL\dual\otimes \omega_{\cD/T}$ and
the isomorphism
\[ \wedge^2 \Omega_{\cX/T}(\log \cX_0)\cong
p^*\cL^\vee\otimes \omega_{\cD/T},
\]
provide us with a two-form
\[
\Theta\in\Gamma(\cX,\wedge^2 \Omega_{\cX/T}(\log\cX_0)).
\]
Its restriction to $\cX_t$ for $t\ne 0$ is the standard two-form
$\theta$; its restriction to $Y_1$ is the one induced by $\theta$
via
$$\Omega_{Y_1}^2\lra \Omega_{Y_1}^2(E)\,\cong\,
\wedge^2 \Omega_{\cX/T}(\log\cX_0)\otimes_{\cO_{\cX}}\cO_{Y_1};
$$
it vanishes along $Y_2$.

Because the restriction $\theta_1=\Theta|_{Y_1}$ vanishes along $E$,
it defines a sheaf homomorphism
$$\theta_1: T_{Y_1}(\log E)\lra \Omega^1_{Y_1}.
$$
Thus for any stable map $f\mh C\to Y_1$ with $f\upmo(E)$ a divisor
in $C$, $\theta_1$ induces a sheaf homomorphism
\begin{equation}\lab{theta2}
f\sta T_{Y_1}(\log E)\lra f\sta \Omega^1_{Y_1}\lra\omega_C;
\end{equation}
this homomorphism will define us the localized relative
GW-invariants of $(Y_1,E)$.

Like the case of ordinary GW-invariants, the localized relative
GW-invariants relies on a cosection of the obstruction sheaf
$\Ob_\cM$ of the moduli of the relative stable morphisms\footnote{
Here as before $d\in H_2(Y_1,\ZZ)$ is the class generated by
$d$-fold of the zero section of $Y_1\to D$; $\eta$ is a partition of
$d$.} $\mxn^{(Y_1,E)}(d,\eta)^{\bullet}$. For us, it is instructive
to see how the two-form $\theta_1$ induces such a cosection. Without
getting into the details of the notion of pre-deformable morphisms
that is essential to the construction of the moduli
$\mxn^{(Y_1,E)}(d,\eta)^\bullet$, we shall describe the obstruction
space and the cosection at a closed point
$\xi\in\mxn^{(Y_1,E)}(d,\eta)^{\bullet}$ that is represented by a
morphism $f\mh C\to Y_1$. In this case, as shown in \cite{Li2}, the
obstruction space $Ob_\xi=\Ob_\cM\otimes\kk(\xi)$ is the cokernel
$$\Ext^1(\Omega_{C}(R),\cO_{C})\lra
H^1(C,f^*T_{Y_1}(\log E))\lra Ob_\xi\lra 0,
$$
where $R$ is the divisor $f\upmo(E)\sub C$; like before, the
homomorphism \eqref{theta2} induces a homomorphism
$$H^1(C, f\sta T_{Y_1}(\log E))\lra H^1(C,\omega_C)=\CC
$$
that lifts to a homomorphism $Ob_\xi\to \CC$; again, this
homomorphism is trivial if and only if $f(C)\sub D_1\sub Y_1$.
Likewise, this construction carries to the family case to give us a
cosection of the obstruction sheaf $\Ob_\cM$; the cosection is
surjective away from those relative stable maps whose image lie in
the zero section of $Y_1$.

We state the property of localized relative GW-invariants of
$(Y_1,E)$ as a proposition; its proof will appear in a separate
paper \cite{KL2}.

\begin{prop}\lab{proprelY1} Let $(Y_1,E)$, let $\theta_1$ and let
$\cM=\mxn^{(Y_1,E)}(d,\eta)$ be as before. Let $\tilde{ev}\mh \cM\to
E^\ell$, $\ell=\ell(\eta)$, be the special evaluation map and let
$E_0=E\cap D_1$. Then the holomorphic two-form $\theta_1$ induces a
canonical cosection $\sigma_1\mh\Ob_\cM\to\cO_\cM$ that is
surjective away from $\tilde{ev}\upmo(E_0^\ell)$. Consequently, its
localized virtual cycle $[\cM]\virt\loc\in H\lsta(E_0^\ell)\sub
H\lsta(E^\ell)$ and the localized relative GW-invariants take values
$$\langle
\tau_{\alpha_1}(\gamma_1)\cdots\tau_{\alpha_n}(\gamma_n)
\rangle^{(Y_1,E),\bullet}_{\chi,\eta,\mathrm{loc}} \in
H\lsta(E^\ell).
$$
\end{prop}

Here we omit the explicit reference to $d$ in the notation for the
localized GW-invariants because $\eta$ is a partition of $d$.

For the pair $(Y_2,E)$, since the form $\Theta$ restricts to zero
along $Y_2$, we shall take its ordinary relative GW-invariants.
Though $Y_2$ is not proper, because $Y_2=\Ao\times\Po$ and $E$ is
one of the $\Ao$ in the product, the special evaluation morphism
$$\tilde{ev}: \mxn^{(Y_2,E)}(d,\eta)^\bullet\lra E^\ell
$$
is proper. Consequently, the relative GW-invariants of $(Y_2,E)$ is
well-defined and take values
$$\langle
\tau_{\alpha_1}(\gamma_1)\cdots\tau_{\alpha_n}(\gamma_n)
\rangle^{(Y_2,E),\bullet}_{\chi,\eta} \in \BM(E^\ell).
$$

In \cite{KL2}, we shall use the intersection pairing
$$\star: H\lsta(E^l,\QQ)\times H\lsta^{BM}(E^l,\QQ)\lra\QQ
$$
to relate the localized GW-invariants of $X$ with the pairings of
the localized relative GW-invariants of $(Y_1,E)$ and the relative
GW-invariants of $(Y_2,E)$. For any partition $\eta$ we define
$\mathbf{m}(\eta)=\prod \eta_i$.

\begin{theo}[\cite{KL2}]
For any integers $\alpha_1,\cdots,\alpha_n\in\ZZ^{\geq 0}$, any
splitting $n_1+n_2=n$ and classes $\gamma_{i(\leq n_1)}\in H^{\geq
1}(Y_1)$ and $\gamma_{j(>n_1)}\in H^{\geq 1}(Y_2)$, we have
\[\langle
\prod_{j=1}^n\tau_{\alpha_j}(\gamma_j)
\rangle^{X,\bullet}_{\chi,d,\mathrm{loc}}=
\sum\frac{\mathbf{m}(\eta)}{|\Aut(\eta)|}\cdot\langle
\prod_{j=1}^{n_1}\tau_{\alpha_j}(\gamma_j)
\rangle^{(Y_1,E),\bullet}_{\chi_1,\eta,\mathrm{loc}} \star \langle
\prod_{j=n_1+1}^n\tau_{\alpha_j}(\gamma_j)
\rangle^{(Y_2,E),\bullet}_{\chi_2,\eta}.
\]
\end{theo}

Here the summation is taken over all integers $\chi_1$ and $\chi_2$,
partitions $\eta$ of $d$ subject to the constraint
\begin{equation}\lab{1.6}
\chi_1 +\chi_2-l(\eta )=\chi.
\end{equation}
The $\Aut(\eta)$ is the subgroup of permutations in $S_{\ell(\eta)}$
that fixes $\eta$.

\subsection{Low degree GW-invariants with descendants}

For a smooth projective surface $S$ with $p_g>0$, in the previous
section we have computed its GW-invariants \emph{without} descendant
insertions; in this section we shall compute its degrees one and two
GW-invariants with \emph{descendants}.

{According to Conjecture \ref{reduct}, we only need consider} the
localized GW-invariants
\begin{equation}\lab{des}
\langle\prod_{i=1}^n\tau_{\aa_i}(\gamma)\rangle^{X,\bullet}_{\chi,d[D],\mathrm{loc}}
=\langle \tau_{\aa_1}(\gamma_1)\cdots \tau_{\aa_n}(\gamma_n)\rangle
^{X,\bullet}_{\chi,d[D],\mathrm{loc}},\quad \text{for }\gamma_i\in
H^*(X),
\end{equation}
of the total space $X$ of a theta characteristic $L$ over a smooth
projective curve $D$ of genus $h:=K_S^2+1$ and $h^0(L)\equiv
\chi(\cO_S)$ mod $2$. Following the convention, since (\ref{des}) is
possibly non-trivial only when
\begin{equation} \lab{last0}
-\chi=d\,K_S^2+\sum_{i=1}^n\aa_i,\quad \aa_i\in \ZZ_{\ge 0},
\end{equation}
we shall omit the reference to $\chi$ in the notation of (\ref{des})
with the understanding that it is given by (\ref{last0}).

Let $\gamma\in H^2(D,\ZZ)$ be the Poincar\'e dual of a point in $D$.
The main result of this section is the following theorem,
conjectured by Maulik and Pandharipande \cite[(8)-(9)]{MaulikPand}

\begin{theo}\lab{MPformthm} Let $X\to D$ and $h=g(D)$ be as before. Then the degree
one and two GW-invariants with descendants are
\begin{equation}\lab{last1} \langle
\prod_{i=1}^n\tau_{\aa_i}(\gamma)\rangle^{X,\bullet}_{[D],\mathrm{loc}}=
(-1)^{h^0(L)}\prod_{i=1}^n\frac{\aa_i!}{(2\aa_i+1)!}(-2)^{-\aa_i};
\end{equation}
\begin{equation}\lab{last2} \langle
\prod_{i=1}^n\tau_{\aa_i}(\gamma)\rangle^{X,\bullet}_{2[D],\mathrm{loc}}
= (-1)^{h^0(L)}\,
2^{h+n-1}\prod_{i=1}^n\frac{\aa_i!}{(2\aa_i+1)!}(-2)^{\aa_i}.
\end{equation}
\end{theo}

The first identity will follow directly from Corollary \ref{twist}
and the second will be proved using degeneration.

We begin with the degree one case. Let $-\chi=K_S^2+\sum\alpha_i$,
$$\cM=\mcn(X,[D])\bul\and \cN=\cM_{\chi,n}(D,[D])\bul ;
$$
let $f\mh\cC\to D$ with $\pi\mh\cC\to \cN$ be the universal family.
Because maps in $\cN$ has degree one over $D$, a direct application
of the base change property shows that $\pi_*f^*L\cong
H^0(L)\otimes\cO_{\cN}$ and that $R^1\pi_*f^*L$ is locally free. By
Corollary \ref{twist},
\[
\langle
\prod_{i=1}^n\tau_{\aa_i}(\gamma)\rangle^{X,\bullet}_{[D],\mathrm{loc}}=
(-1)^{h^0(L)}\int_{[\cN]\virt} \prod
\psi_i^{\aa_i}ev_i^*(\gamma)\cap c_{top}(-\pi_!f^*L),
\]
is a twisted GW-invariant of $D$ with sign $(-1)^{h^0(L)}$. From
\cite{FaberPand-HodgeInt}, we can readily evaluate them and obtain
\eqref{last1}. We omit the straightforward computation here.

\bigskip

We next prove \eqref{last2}. We begin with a few special cases.

\begin{lemm}[$h=0$ case]
Let $Y_0$ be the total space of $\cO_{\PP^1}(-1)$. Then
\begin{equation}\lab{last6} \langle
\prod_{i=1}^n\tau_{\aa_i}(\gamma)\rangle^{Y_0,\bullet}_{2[\Po],\mathrm{loc}}
= 2^{n-1}\prod_{i=1}^n\frac{\aa_i!}{(2\aa_i+1)!}(-2)^{\aa_i}.
\end{equation}
\end{lemm}
\begin{proof}
Since $\PP^1\subset \cO_{\PP^1}(-1)$ is rigid, the moduli space of
stable maps to $\cO_{\PP^1}(-1)$ is proper and the localized
GW-invariant coincides with the twisted GW-invariant of $\PP^1$ by
Corollary \ref{twist}. Hence, \eqref{last6} follows from the
differential equations for the twisted invariants in
\cite{FaberPand-HodgeInt}.
\end{proof}

\begin{lemm} \lab{lem3} We have
\begin{equation}\lab{last3} \langle
\tau_1(\gamma)\rangle^{X,\bullet}_{2[D],\mathrm{loc}}
=(-1)^{h^0(L)}\left( \frac{2^{h}}{-3}\right)
\end{equation}
\end{lemm}

We postpone the proof of this lemma until later.

\begin{lemm}\lab{last4}
Let $(Y_1,E)$ and $(Y_2,E)$ be the relative pairs resulting from the
degeneration constructed in the previous subsection. Then
\[
\langle 1
\rangle^{(Y_1,E),\bullet}_{(1,1),\mathrm{loc}}=(-1)^{h^0(L)}\,
2^{h}\, [pt^2] \and \langle \tau_1(\gamma) \rangle^{(
Y_2,E),\bullet}_{(1,1)}\star [pt^2]= -1/6.
\]
\end{lemm}
\begin{proof}
We first look at the first identity. It is easy to see, from the
construction of localized relative invariants (Proposition
\ref{proprelY1}), that $\langle 1
\rangle^{(Y_1,E),\bullet}_{(1,1),\mathrm{loc}}$ is a scalar multiple
of $[pt^2]$. Thus, an easy virtual dimension counting shows that the
only stable maps that contribute to this invariant must have
$-\chi=2(h-1)$. By \eqref{last0}, the composites of these stable
maps with $p\mh X\to D$ are \'etale covers of $D$. Hence the
(relevant) moduli space of relative stable maps to $(Y_1,E)$ is a
disjoint union of $2^{2h}$ vector spaces, each consists of all
liftings of an \'etale cover of $D$. By our proof of Proposition
\ref{fundtwist}, we obtain
\[
\langle 1
\rangle^{(Y_1,E),\bullet}_{(1,1),\mathrm{loc}}=\sum_{u:2\text{-fold
\'etale covers of }D}(-1)^{h^0(u^*L)}\,[pt^2].
\]
It is known that \'etale double covers of $D$ are parameterized by
the set of order $2$ line bundles on $D$, and exactly
$2^{h-1}(2^{h}+1)$ of them satisfy $h^0(u^*L)\equiv h^0(L)$ mod $2$
\cite{Harris}. Therefore we have
\[
\langle 1 \rangle^{(Y_1,E),\bullet}_{(1,1),\mathrm{loc}}=
(-1)^{h^0(L)} \left( 2^{h-1}(2^{h}+1) -2^{h-1}(2^{h}-1)
\right)[pt^2] =(-1)^{h^0(L)} 2^{h}[pt^2].
\]
This proves the first equation.

Since $Y_2$ is the total space of the trivial line bundle over
$\PP^1$, any stable map in
$\cM_{\chi,n}^{Y_2,E}(2[D],(1,1))^\bullet$ with two distinct
intersection points with $E$ has two irreducible components, one
with the marked point and the other without. Therefore, because
$$
\langle \tau_1(\gamma) \rangle^{( Y_2, E)}_{(1)} \star [pt]= \langle
\tau_1(\gamma) \rangle^{Y_0,\bullet}_{[\PP^1],\mathrm{loc}}
=-\frac1{12} \and \langle 1 \rangle^{( Y_2, E)}_{(1)} \star [pt]=1,
$$
we have
\[
\langle \tau_1(\gamma) \rangle^{( Y_2, E),\bullet}_{(1,1)}\star
[pt^2] =\left( \bigl\langle \tau_1(\gamma)
\bigr\rangle^{(Y_2,E)}_{(1)} \star [pt] \right) \left( \langle
1\rangle^{(Y_2,E)}_{(1)} \star [pt] \right) +\qquad\qquad
\]
\[
\qquad\qquad\qquad\qquad\qquad\quad+\left( \langle
1\rangle^{(Y_2,E)}_{(1)} \star [pt] \right) \left( \langle
\tau_1(\gamma) \rangle^{(Y_2,E)}_{(1)} \star [pt] \right) = -\frac16
\]

\end{proof}

We next prove \eqref{last2}, assuming \eqref{last3}. By the
degeneration formula, we have
\begin{equation}\lab{last5} \langle
\prod_{i=1}^n\tau_{\aa_i}(\gamma)
\rangle^{X,\bullet}_{2[D],\mathrm{loc}} = \frac12\ \langle 1
\rangle^{(Y_1,E),\bullet}_{(1,1),\mathrm{loc}} \star \langle
\prod_{i=1}^n\tau_{\alpha_i}(\gamma) \rangle^{( Y_2,
E),\bullet}_{(1,1)}+\qquad
\end{equation}
$$ \qquad\qquad\qquad\qquad\qquad+ 2\  \langle 1
\rangle^{(Y_1,E),\bullet}_{(2),\mathrm{loc}} \star \langle
\prod_{i=1}^n\tau_{\alpha_i}(\gamma) \rangle^{( Y_2,
E),\bullet}_{(2)}.
$$
In particular, from \eqref{last3} and Lemma \ref{last4}, we have
\[
(-1)^{h^0(L)}\left( \frac{2^{h}}{-3}\right) =\langle
\tau_1(\gamma)\rangle^{X,\bullet}_{2[D],\mathrm{loc}} = \frac12\
(-1)^{h^0(L)}\left( \frac{2^{h}}{-6}\right)+\, 2\langle 1
\rangle^{(Y_1,E),\bullet}_{(2),\mathrm{loc}} \star \langle
\tau_{1}(\gamma) \rangle^{( Y_2, E),\bullet}_{(2)}.
\]
Comparing this with the case where $D=\Po$ ($h=0$), we see that the
relative invariants $\langle 1
\rangle^{(Y_1,E),\bullet}_{(1,1),\mathrm{loc}} $ and $\langle 1
\rangle^{(Y_1,E),\bullet}_{(2),\mathrm{loc}} $ are exactly those for
$D=\Po$, multiplied by $(-1)^{h^0(L)}\ 2^{h}$. Therefore by
\eqref{last5} and \eqref{last6}, we have
\[
\langle \prod_{i=1}^n\tau_{\aa_i}(\gamma)
\rangle^{X,\bullet}_{2[D],\mathrm{loc}} = (-1)^{h^0(L)}\ 2^{h}\
\langle \prod_{i=1}^n\tau_{\aa_i}(\gamma)
\rangle^{Y_0,\bullet}_{2[\Po],\mathrm{loc}}\qquad\qquad\qquad
\]
\[
\qquad\qquad= (-1)^{h^0(L)}\ 2^{h+n-1}\ \prod_{i=1}^n
\frac{\aa_i!}{(2\aa_i+1)!}(-2)^{\aa_i}.
\]
This proves Theorem \ref{MPformthm}.

\subsection{Proof of Lemma \ref{lem3}}

It remains to prove \eqref{last3}. We first claim that for
$-\chi=2(h-1)+1$, the moduli space
$\cN=\cM_{\chi,1}(D,2[D])^\bullet$ has exactly $2^{2h}+1$ connected
components. Obviously, one of them consists of all double covers of
$D$ branched at two points; we denote this component by $\cN_0$.
Each of the other components is distinguished by an \'etale double
cover $u\mh C\to D$: elements in this component are formed by adding
genus one ghost components to $u$. We denote this component by
$\cN_u$ and denote by $\cM_u$ the corresponding component in
$\cM=\cM_{\chi,1}(X,2[D])$. Because there are $2^{2h}$ \'etale
double covers of $D$, there are $2^{2h}$ of such components.

It is easy to see that the contribution to the localized
GW-invariant $\langle \tau_1(\gamma)
\rangle^{X,\bullet}_{2[D],\mathrm{loc}}$ from the component $\cM_u$
is
\[
\langle \tau_1(\gamma)
\rangle^{X,\bullet}_{2[D],\mathrm{loc}}\bigl[\cM_u\bigr]
=(-1)^{h^0(u^*L)}\ \left( -\frac1{12}\right).
\]
Here the sign $(-1)^{h^0(u^*L)}$ is due to Proposition 3.5 and the
factor $-1/12$ is from the formula \eqref{last1} for the degree one
case. Because exactly $2^{h-1}(2^{h}+1)$ of the $2^{2h}$ \'etale
double covers $u:C\to D$ satisfy $h^0(u^*L)\equiv h^0(L)$ mod $2$,
the total contribution to $\langle \tau_1(\gamma)
\rangle^{X,\bullet}_{2[D],\mathrm{loc}}$ from these irreducible
components is
\[
(-1)^{h^0(L)}\, \left( -\frac1{12} \right)\cdot\bl 2^{h-1}(2^{h}+1)
-
2^{h-1}(2^{h}-1)\br = (-1)^{h^0(L)} \left( -\frac{2^{h}}{12}
\right).
\]
Therefore to prove (\ref{last3}), it suffices to show

\begin{lemm}\lab{contri-0}
The contribution to $\langle \tau_1(\gamma)
\rangle^{X,\bullet}_{2[D],\mathrm{loc}}$ from the connected
component $\cM_0$ is
\begin{equation*}\lab{last9} \langle \tau_1(\gamma)
\rangle^{X,\bullet}_{2[D],\mathrm{loc}} \bigl[\cM_0\bigr] =
(-1)^{h^0(L)}\left( \frac{2^{h}}{-3}\right) - (-1)^{h^0(L)}\left(
-\frac{2^{h}}{12} \right) =(-1)^{h^0(L)} (-2^{h-2}).
\end{equation*}
\end{lemm}

Unlike the previous case, we cannot apply Proposition
\ref{fundtwist} directly to evaluate the above quantity because
$\pi_*f^*L$ is not locally free over $\cN_0$. Nevertheless, by a
detailed investigation of its failure to being locally free we can
identify the extra contribution, thus proving the lemma.

It is easy to get the contribution to the twisted GW-invariant from
the component $\cN_0$ of branched double covers. From
\cite{FaberPand-HodgeInt} and \cite{OkounkovPand}, it is
straightforward to deduce that the degree two GW-invariants of $D$
twisted by the top Chern class of $-\pi_!f^*L$ is
\[
\langle \tau_1(\gamma);c_{top}(-\pi_!f^*L)
\rangle^{D,\bullet}_{2,\chi} = \left( h-\frac83 \right)2^{2h-3};
\]
the contribution to the twisted GW-invariant from any of the
irreducible components $\cN_u$ is $-\frac1{12}$. Therefore, the
contribution to the twisted GW-invariant from $\cN_0$ is
\begin{equation}\lab{last10}
\langle \tau_1(\gamma);c_{top}(-\pi_!f^*L)
\rangle^{D,\bullet}_{2,\chi}\bigl[\cN_0\bigr] =\left( h-\frac83
\right)\, 2^{2h-3} - 2^{2h}\, \left( -\frac1{12}\right) = \left( h-2
\right)\, 2^{2h-3}.
\end{equation}

Our next step is to take $D$ to be a \emph{hyperelliptic} curve with
$\delta:D\to \PP^1$ double cover; we then cut down the space $\cN_0$
by the insertion $\tau_1(\gamma)=\psi_1 ev_1^*(\gamma)$. As is
known, a double cover $u:C\to D$ branched at two points $q_1,q_2\in
D$ is characterized by a line bundle $\xi$ on $D$ satisfying
$\xi^2\cong \cO_D(q_1+q_2)$: the curve $C$ is a subscheme of the
total space of $\xi$ defined by $\{t\in \xi\,|\, t^2=v\}$ for $v$ a
section in $H^0(\cO_D(q_1+q_2))$ vanishing at $q_1+q_2$; the map is
that induced by the projection $\xi\to D$.

We now let
$$f: \cC\to D,\quad \pi: \cC\to \cN_0,\quad
s: \cN_0\to \cC \ \ \text{the section of marked points}
$$
be the universal family of $\cN_0$; we let $\ev_1\mh \cN_0\to D$ as
before be the evaluation morphism.
We can cut down $\cN_0$ by $\psi_1 ev_1\sta(\gamma)$ as follows: we
pick a general point $q\in D$ and form the subscheme $ev_1^{-1}(q)$;
it represents the class $ev_1^*(\gamma)$.
For $\psi_1$, we observe that the natural homomorphism
$s^*f^*\Omega_D\to s^*\omega_{\cC/\cN_0}$ induces a section
\begin{equation}\lab{section}
\phi\in\Gamma\bl\cN_0,s^*(\omega_{\cC/\cN_0}\otimes
f\sta\Omega\dual_D)\br.
\end{equation}
This section vanishes at a map $u=(\xi, q_1+q_2)\in\cN_0$ if and
only if the marked point is one of the two branched points $q_1$ and
$q_2$. For convenience, we denote $\cZ=\{\phi=0\}$; for $q\in D$, we
denote $\cZ_q=\cZ\times_D q$ with $\cZ\to D$ induced by evaluation
map $ev_1$. This way, the class $\psi_1 ev_1\sta(\gamma)$ is
represented by the substack
$$\cM_{0,q}=\cM\times_{\cN_0}\cZ_q.
$$

\begin{lemm}
The stack $\cM_{0,q}$ has an induced perfect obstruction theory; its
obstruction sheaf has a cosection consistent with that of $\cM_0$;
the degree of its localized virtual cycle equals the contribution
$\bigl\langle\tau_1(\gamma)\bigr\rangle^{X,\bullet}_{2[D],
loc}[\cM_0]$.
\end{lemm}

\begin{proof}
Our first step is to argue that there are line bundles $L_1$ and
$L_2$ on $\cM_0$ and their respective sections $s_1$ and $s_2$ so
that $(s_1=s_2=0)=\cM_{0,q}$. Indeed, because of our construction
that $\cZ\sub\cN_0$ is a Cartier divisor and $\cZ_q=\cZ\times_D q$,
$\cZ_q\sub\cN_0$ is cut out by two sections of two line bundles on
$\cN_0$. Then since $\cM_{0,q}=\cM_0\times_{\cN_0}\cZ_q$,
$\cM_{0,q}$ is cut out by the vanishing of the pullback of these
sections. This proves the claim.

We let the two pullback line bundles be $L_1$ and $L_2$ and let the
two pullback sections be $s_1$ and $s_2$. According to \cite{Li2},
$(s_1,s_2)$ induces an obstruction theory of
$(s_1=s_2=0)=\cM_{0,q}$; its obstruction sheaf $\Ob_{\cM_{0,q}}$ and
the obstruction sheaf $\Ob_{\cM_0}$ of $\cM_0$ fits into the exact
sequence
\begin{equation}\lab{induced}
\lra\cO_{\cM_{0,q}}(L_1)\oplus \cO_{\cM_{0,q}}(L_2)\lra
\Ob_{\cM_{0,q}}\lra
\Ob_{\cM_{0}}\otimes_{\cO_{\cM_{0}}}\cO_{\cM_{0,q}}\lra 0;
\end{equation}
the obstruction theory of $\cM_{0,q}$ is perfect as well.

In \cite[Lemma 4.6]{Li2}, it is proved that
$$[\cM_{0,q}]\virt=c_1(L_1)\cup c_1(L_2)[\cM_0]\virt.
$$
To apply to the localized GW-invariants, we need to define the
localized virtual cycle $[\cM_{0,q}]\virt\loc$; show that applying
the localized first Chern class $c_1(L_i,s_i)$ we will have
\begin{equation}\lab{cyc-comp}
[\cM_{0,q}]\virt\loc=c_1(L_1,s_1)\cup c_1(L_2,s_2)[\cM_0]\virt\loc.
\end{equation}
For this, let $\sigma\mh\Ob_{\cM_0}\to\cO_{\cM_0}$ be the cosection.
Because of \eqref{induced}, $\Ob_{\cM_{0,q}}$ has an induced
cosection, say $\sigma_q$. Since the degeneracy locus
$Z(\sigma)\sub\cM_0$ is proper, $Z(\sigma_q)=Z(\sigma)\cap\cM_{0,q}$
is also proper. Further, by combining the argument of \cite[Lemma
4.6]{Li2} and the proof of Proposition \ref{constancy}, one proves
the identity \eqref{cyc-comp}. Since the argument is routine, we
shall omit the details here. This proves the lemma.
\end{proof}

In light of this lemma, we need to investigate the stack structure
of $\cM_{0,q}$. According to Lemma \ref{yh5.5}, it suffices to
investigate the locus $\Lambda\sub \cZ$ near which the sheaf
$R^1\pi\lsta f\sta L$ on $\cN_0$ is non-locally free. For a double
cover $u\mh C\to D$ in $\cZ$ given by $(\xi, q_1+q_2)$, since
\begin{equation}\lab{splitL}
H^0(C,u^*L)=H^0(D,L)\oplus H^0(D,L\otimes \xi^{-1}),
\end{equation}
it is easy to see that $R^1\pi\lsta f\sta L$ fail to be locally free
at $u$ exactly when $h^0(L\otimes \xi^{-1})\ne 0$. Let
$$\eta\ne
0\in H^0(D,L\otimes \xi^{-1})
$$
and let $H=\delta^*\cO_{\PP^1}(1)$.
Then as $L^2\cong K_D=(h-1)H$ and $\xi^{\otimes 2}=\cO(q_1+q_2)$, we
have
\begin{equation}\lab{eq-5}
(h-1)H=q_1+q_2+2\eta\upmo(0),
\end{equation}
which, because $D$ is hyperelliptic, is possible only if either
$q_1=q_2$ or they are conjugate to each other. In particular, when
$u\in \Lambda\cap\cZ_q$, one of $q_i$ must be $q$; thus
$\Lambda\cap\cZ_q$ is finite.

Were this intersection empty, by Corollary \ref{twist}, the
localized GW-invariant \eqref{last9} would have been equal to the
twisted GW-invariant multiplied by $(-1)^{h^0(L)}$. To find the
correction caused by these exceptional points, we need a detailed
analysis of the virtual normal cone of the moduli space $\cM_0$ near
the fibers over $\Lambda$. Since the cases $q_2=q_1$ and $=\bar q_1$
($\bar q_1$ is the conjugate point of $q_1$) require independent
analysis, we shall investigate them separately. We denote by
$\Lambda\dpri\sub\Lambda$ the subset of elements associated to $q_1=
q_2$; we denote $\Lambda\pri=\Lambda-\Lambda\dpri$.

For general $q\in D$, it is easy to enumerate the set
$\Lambda\pri\cap\cZ_q$. Let $u=(\xi,q+\bar q)$ be any point in this
set. Since $q+\bar q=H$, (\ref{eq-5}) reduces to
$2\eta\upmo(0)=(h-2)H$. Let
\begin{equation}\lab{def-r}
r=h^0(L\otimes \xi\upmo)-1=h^1(L\otimes \xi\upmo)-2,
\end{equation}
which ranges between $0$ and $(h-2)/2$. Since $D$ is hyperelliptic,
$\eta\upmo(0)$ must be $p_1+\cdots +p_{h-2}$ for points $p_i$ on $C$
so that $p_{i+r}=\bar p_i$ for $i\leq r$ and $p_j$, $j\geq 2r+1$,
are distinct ramification points of $\delta:D\to \PP^1$. The $r$ in
\eqref{def-r} decomposes $\Lambda\pri\cap\cZ_q$ into union $\cup
\Lambda\pri_r$; elements in $\Lambda\pri_r$ are uniquely determined
by the distinct ramified points $p_{j>2r}$ since
\begin{equation}\lab{rel-3}
\xi=L(-p_1-\cdots -p_{h-2})=L-rH-p_{2r+1}-\cdots - p_{h-2}.
\end{equation}
Therefore, $\Lambda\pri_r$ consists of $\binom{2h+2}{h-2-2r}$
elements, where $2h+2$ is the number of ramification points of
$\delta$.

We next claim that the scheme structure of
$\cM_{0,q}=\cM_0\times_{\cN_0}\cZ_q$ near the fiber over
$u=(\xi,q+\bar q)\in \Lambda\pri_{r}$ is (analytically) isomorphic
to
\begin{equation}\lab{eq-6}
\bA^l\times\left\{ zw_1=z^3w_2=\cdots =z^{2r+1}w_{r+1}=0
\right\}\sub \bA^{l+r+2} ,\quad l=h^0(L).
\end{equation}
Since $\cZ_q$ is smooth near $u$, the local defining equation of
$\cM_{0,q}$ is determined by a locally free resolution of $\bar\pi_!
\bar f\sta L$ for $\bar f\mh \bar\cC\to D$ and $\bar\pi\mh\bar\cC\to
\cZ_q$ the restrictions of $f$ and $\pi$ to $\cZ_q$. On the other
hand, from (\ref{splitL}), we see immediately that $\bar\pi_*\bar
f^*L=H^0(L)\otimes\cO$. By Riemann-Roch, away from
$\cM_{0,q}\times_{\cZ}\Lambda$, $R^1\bar\pi_*\bar f^*L$ is locally
free and has rank $l+1$. Therefore, $R^1\bar\pi_*\bar f^*L$ is a
direct sum of its torsion free part $\cR$ and its torsion part.
Since $h^1(L\otimes \xi^{-1})=r+2$ by Riemann-Roch \cite{ACGH}, in a
formal neighborhood of $u\in \cZ_q$ with $z$ the local coordinate of
$\cZ_q$ at ${u}$, we can find positive integers $\aa_1\le
\aa_2\le\cdots \le \aa_{r+1}$ and express the torsion part of
$R^1\bar\pi_*\bar f^*L$ as
\begin{equation}\lab{torsion}
\CC[[z]]/(z^{\aa_1})\oplus \cdots \oplus\CC[[z]]/(z^{\aa_{r+1}}).
\end{equation}
Thus a (locally free) resolution of $\bar\pi_!\bar f^*L$ can be
chosen as
\[
\mathrm{diag}(z^{\aa_1},\cdots,z^{\aa_{r+1}}):
\bigoplus_{i=1}^{r+1}\cO\lra \left(
\bigoplus_{i=1}^{r+1}\cO\right)\oplus \cR.
\]
By Lemma \ref{yh5.5}, if we let $(z_1,\cdots, z_l, w_1,\cdots,
w_{r+1})$ be the coordinate for the vector space
$H^0(C,f^*L)=H^0(L)\oplus H^0(L\otimes\xi^{-1})$, the local defining
equation of $\cM_{0,q}$ near fibers over $u$ can be chosen as
\[
z^{\aa_1}w_1=z^{\aa_2}w_2=\cdots =z^{\aa_{r+1}}w_{r+1}=0.
\]

It remains to determine the integers $\aa_i$. For this we need to
investigate whether a line $\CC \eta\sub H^0(L\otimes\xi^{-1})$ can
be extended to a submodule
\begin{equation}\lab{ext5}
\CC[z]/(z^{n})\sub \pi\lsta f\sta L/z^n \pi\lsta f\sta L.
\end{equation}
Let
\[
H^0(L\otimes\xi^{-1})=F_{r+1}\supseteq F_r\supseteq\cdots\supseteq
F_1\supseteq F_0=\{0\}
\]
be a flag with $F_{r+1-k}=H^0(L\otimes \xi^{-1}(-kq))$. It is a
complete flag since $h^0(L\otimes \xi^{-1}(-kq))=\max(r+1-k,0)$,
using \eqref{rel-3}.
%
%
For $0\le k\le r$, let $\eta\in F_{r+1-k}-F_{r-k}$ be a general
element and $(p_i)$ be the zeros of $\eta$. Then after reshuffling
if necessary, we have
\[
p_1=\cdots =p_k=q,\ p_{r+1}=\cdots =p_{r+k}=\bar q,\ p_{k+1},\cdots,
p_r, p_{r+k+1},\cdots, p_{2r}\notin \{q,\bar q\}.
\]
Suppose $\eta$ fits into a $\CC[z]/(z^n)$-modules as in
\eqref{ext5}, then there are morphisms
$$\bq, \ {\bold q}\pri\and {\bold p}_i: \spec\CC[z]/(z^n)\lra D,
$$
which extend the points $q,\bar q$ and $p_i$ respectively, such that
$\bq$ is constant, $\bq\pri$ has non-vanishing first order
variation, and that as families of divisors on $D$ parameterized by
$\spec\CC[z]/(z^n)$, the relation \eqref{eq-5} holds:
$$(h-1)H=\bq+\bq\pri+2\bold p_1+\cdots+2\bold p_{h-2}.
$$
Because $\bar q$ is not a branched point of $\delta\mh D\to\Po$, our
local coordinate $z$ for $\cZ_q$ at $u$ can be thought of as a local
coordinate for $D$ near $\bar q$ and also a local coordinate for
$\PP^1$ via $\delta$. Without loss of generality, we can choose
$\bq\pri$ so that $\delta\circ\bq\pri(z)=z$. Suppose
$\delta\circ\bold p_i=p_i(z)$. Then the above identity on divisors
over $\spec\CC[z]/(z^n)$ is equivalent to
\begin{equation}\lab{last20}
w\prod_{i=1}^k(w-p_i(z))^2\equiv
(w-z)\prod_{i=1}^k(w-p_{r+i}(z))^2\mod z^n,
\end{equation}
for a formal variable $w$.

Shortly, we shall show that \eqref{last20} is solvable if and only
if $n\le 2k+1$. Once this is done, we see that because $\dim F_k=k$,
we have $\alpha_1=1,\,\alpha_2=3,\,\ldots\,,\alpha_r=2r+1$. This
will provide us the structure result of $\cM_{0,q}$ over an element
in $\Lambda\pri_r$.

We now prove \eqref{last20} is solvable if and only if $n\leq 2k+1$.
We first compare the constant coefficients in $w$ of \eqref{last20};
we obtain
\begin{equation}\lab{last21} z\prod_{i=1}^k p_{r+i}^2(z)\equiv 0 \mod
z^n .
\end{equation}
To proceed, we shall show that for $n=2k+1$, \eqref{last20} holds
true if and only if $p_{r+i}(z)\equiv_{(z^2)}\! c_iz$ for uniquely
determined complex numbers $c_i\ne 0$, $1\le i\le k$. (Here we use
$\equiv_{(z^n)}$ to mean equivalence modulo $z^n$.) Together with
\eqref{last21}, this immediately implies the claim.

Let $n=2k+1$; 
let $f$, $g$ and $h$ be defined by
$$f=\prod_{i=1}^k(1-p_i(z)/w)=\sum_{j=0}^kA_jw^{-j},\quad
g=\prod_{i=1}^k(1-p_{r+i}(z)/w)=\sum_{j=0}^kB_jw^{-j}
$$
and $h=\sqrt{1-z/w}\cdot g=\sum_{j=0}^\infty C_jw^{-j}$, with $A_i$,
$B_i$ and $C_i$ analytic functions of $z$. Dividing \eqref{last20}
by $w^{2k+1}$, we see immediately that that $C_j\equiv_{(z^n)}\!A_j$
for $0\le j\le k$ and that $h^2-f^2=2f(h-f)+(h-f)^2$ has no terms
$w^{-j}$ with $j\le 2k$, modulo $z^n$ as usual. Since $f$ is monic
and $h-f$ has no terms $w^{-j}$ with $j\le k$, $h-f$ has no terms
$w^{-j}$ with $j\le 2k$. This implies that $C_j\equiv_{(z^n)}\!0$
for $k<j\le 2k$. If we let $\sqrt{1-z/w}=\sum D_jw^{-j}$, we obtain
\[
0\equiv_{(z^n)}\!C_j = \sum_{i=0}^kD_{j-i}B_i=
D_j+\sum_{j=1}^kD_{j-i}B_i
\]
for $k<j\le 2k$, and thus we have a matrix equation
\begin{equation}\lab{last21g}
(G_1\,\, G_{2}\,\, \cdots \,\, G_k) B
\equiv_{(z^n)}\! -G_{k+1},
\end{equation}
where $G_j$ is the column vector $(D_j,D_{j+1},\cdots, D_{j+k-1})^t$
and $B$ is the column vector $(B_k,B_{k-1},\cdots,B_1)^t$. Since
$D_j=-2^{1-2j} \frac{1}{j} \binom{2j-2}{j-1} z^j$, the determinant
of the Hankel matrix $(G_1\,\, G_{2}\,\, \cdots \,\, G_k)$ is
$[(-1)^k/2^{2k^2-k}]z^{k^2}$ and the determinant of $(G_2\,\,
G_{3}\,\, \cdots \,\, G_{k+1})$ is $[(-1)^k/2^{2k^2+k}]z^{k^2+k}$.
(See \cite{Radoux} for the computation of these Hankel
determinants.) To prove the solvability of \eqref{last20} for
$n=2k+1$, we can replace $\equiv_{(z^n)}$ by the honest equality. By
Cramer's rule, the matrix equation $(G_1\,\, G_{2}\,\, \cdots \,\,
G_k) B = -G_{k+1}$ has a unique solution $B_k=(-4)^{-k}z^k$ and
$B_j=\beta_j z^j$ for some $\beta_j\in \CC$, $j\le k$. The equation
$A_j=C_j$ for $j\le k$ implies $A_j=\gamma_jz^j$ for some
$\gamma_j\in \CC$. Therefore $p_{r+i}(z)=c_iz$ for some $c_i\ne 0$
solves \eqref{last20} and \eqref{last21}.

To see the insolvability of \eqref{last20} for $n>2k+1$, we observe
from \eqref{last21g} that the least order terms of $B_j$ are
uniquely determined as above and thus $p_{r+i}(z)\equiv_{(z^2)}c_iz$
for $c_i\ne 0$. Then \eqref{last21} cannot be satisfied. This
completes our description of the formal neighborhoods of exceptional
points in $\Lambda\pri_r$.
%

The structure of $\Lambda\dpri$ is similar. First,
$\Lambda\dpri\cap\cZ_q$ is the disjoint union of $\Lambda\dpri_{r}$
for $0\le r\le \frac{h-3}2$; each $\Lambda\dpri_{r}$ has $
\binom{2h+2}{h-3-2r}$ elements; the local defining equation for the
stack $\cM_{0,q}$ near a fiber over a point in $\Lambda\dpri_{r}$ is
isomorphic to
\begin{equation}\lab{II}
\bA^l\times\left\{z^{2}w_1=z^{4}w_2=\cdots
=z^{2(r+1)}w_{r+1}=0\right\}\sub\bA^{l+r+2}.
\end{equation}

Combined, we see that $\cM_{0,q}$ has one irreducible component
$V_0$ that dominates $\cZ_q$: it is the vector bundle
$\bar\pi\lsta\bar f\sta L$ over $\cZ_q$; other irreducible
components lie over points in $\Lambda\cap \cZ_q$: one for each
$u\in \Lambda\pri_r\cap\cZ_q$---we denote this component by
$V_u\sub\bar p\upmo(u)$, where $\bar p\mh\cM_{0,q}\to\cZ_q$ is the
projection as before. Finally, we comment that points in $V_u$ have
automorphism groups $\ZZ_2$ since the only marked points are branch
points for these stable maps.

We now prove Lemma \ref{lem3}. We let $E$ be a vector bundle on
$\cM_{0,q}$ so that its sheaf of sections surjects onto
$\Ob_{\cM_{0,q}}$; we let $W\sub E$ be the virtual normal cone of
the obstruction theory of $\cM_{0,q}$. Of all the irreducible
components of $W$, one dominates $V_0$. It is a sub-bundle of
$E|_{V_0}$; its normal bundle in $E|_{V_0}$ is the pullback of the
torsion free part $\cR$ of $R^1\bar \pi\sta\bar f\sta L$. Therefore,
the contribution to the localized GW-invariant of this component is
$(-1)^l$ times the first Chern class of $\cR$. Here $l=h^0(L)$ is
the dimension of the fiber of $V_0\to\cZ_q$; the sign $(-1)^l$ is
due to the reason similar to the proof of Proposition
\ref{fundtwist}; the degree of $\cR$ is the difference of
$c_1(-\bar\pi_!\bar f^*L)$ and the total degree of the torsion part
of $R^1\bar\pi\lsta\bar f^*L$. By \eqref{eq-6}, the latter at a
point in $\Lambda\pri_r\subset \Lambda\pri$ is \eqref{torsion} with
$\alpha_i=2i-1$;
thus the total degree of the torsion part lying over $\Lambda\pri$
is
\[
\sum_{r=0}^{[\frac{h-2}{2}]} \frac12\cdot
a_{2r}\cdot\binom{2h+2}{h-2-2r},\quad a_{2r}=1+3+5+\cdots +(2r+1).
\]
Here the factor $\frac12$ is from the trivial $\ZZ_2$ action.
Similarly, the degree of the torsion part lying over $\Lambda\dpri$
is
\[
\sum_{r=0}^{[\frac{h-3}2]}\frac12\cdot
a_{2r+1}\cdot\binom{2h+2}{h-3-2r},\quad a_{2r+1}=2+4+6+\cdots
+(2r+2).
\]
Hence, by Proposition \ref{fundtwist} and \eqref{last10}, the
contribution to the localized GW-invariant of this irreducible
component is
\[
(-1)^{l}\Bigl[(h-2)2^{2h-3} - \sum_{j=0}^{h-2}\binom{2h+2}{h-2-j}a_j
\Bigr].
\]

The other irreducible components are supported over $V_u$ for
$u\in\cZ_q\cap \Lambda$. Let $m=\rank E$. Over a $u\in \cZ_q\cap
\Lambda\pri_r$, the virtual normal cone $W$ has $(r+1)$ irreducible
components lying over $V_u\sub\bar p\upmo(u)$: they are indexed by
$0\leq i\leq r$; the $i$-th is supported on a rank $m+i-(l+r+1)$
subbundle of $E$, of multiplicity $2i+1$, over
$$V_{u,i}=\bA^l\times\{w_1=\cdots=w_i=0,\, z^{2i+1}=0\}\sub V_u
\sub\bar p\upmo(u).
$$
Thus following the proof of Proposition \ref{fundtwist}, the total
contribution to the localized GW-invariant of these components is
\[
b_{2r}\cdot \frac{(-1)^l}2=\sum_{i=0}^r(-1)^{l+r+1-i}\, (2i+1)\,
\frac12.
\]
As mentioned, the sign is from Proposition \ref{fundtwist}; $(2i+1)$
is the multiplicity; $1/2$ is from the trivial $\ZZ_2$ action.

By the same reason, there are $(r+1)$ irreducible components of $W$
over a point $u\in\Theta_r\dpri$; they are indexed by $0\leq i\leq
r$; the $i$-th is a rank $m+i-(l+r+1)$ sub-bundle of $E$ over
$\bA^{l+r+1-i}$ with multiplicities $2i+2$. The contribution to the
localized GW-invariant of these components is
\[
b_{2r+1}\cdot \frac{(-1)^l}2=\sum_{i=0}^r(-1)^{l+r+1-i}\, (2i+2)\,
\frac12.
\]

Combining the above, the contribution to the localized GW-invariant
of the component $\cN_0$ of branched double covers is
\[
(-1)^{h^0(L)}\Bigl[(h-2)2^{2h-3} -
\sum_{j=0}^{h-2}\binom{2h+2}{h-2-j}\cdot\frac{a_j-b_j}{2} \Bigr].
\]
Now it is an elementary combinatorial exercise to check that this
coincides with the desired $(-1)^{h^0(L)}(-2^{h-2})$. This completes
our proof of Lemma \ref{contri-0}.

\section{Remarks on three-fold case}

We shall conclude our paper by commenting on how our method of
localization by holomorphic two-form can be applied to study
GW-invariants of three-folds.

Let $X$ be a smooth projective three-fold over a smooth projective
surface $S$
$$p:X\lra S;
$$
we assume $S$ has a holomorphic two-form $\theta$ whose vanishing
locus is $D$; $\theta$ pulls back to a holomorphic two-form $\tilde
\theta$ on $X$.
Let $\beta\in H_2(X,\ZZ)$ and
$u:C\to X$ be a stable map such that $u_*[C]=\beta$. Then $u$ is
$\tilde{\theta}$-null only if for any irreducible component $C\pri$
of $C$ its image $p(C\pri)$ is either a point or is contained in
$D$. Now suppose $p_*(\beta)\ne 0\in H_2(S,\ZZ)$; since $C$ is
connected, $u$ is $\tilde{\theta}$-null stable map if and only if
$u(C)\sub Y=p^{-1}(D)$. Here $Y=p^{-1}(D)$ is defined by the
Cartesian square
\[
\xymatrix{ X \ar[d]_p & \,Y\ar@{_(->}[l]_i \ar[d]^q\\
S & \,D\ar@{_(->}[l]_j  }
\]
Hence, the virtual fundamental class of the
moduli space $\mgn(X,\beta)$ is supported in the locus of stable
maps to $Y$. In particular, if $\beta$ is not in the image of
$H_2(Y,\ZZ)$, the GW-invariants vanish.

Now suppose $\beta'=j_*(d[D])$ for some $d\in \ZZ-\{0\}$. Under
favorable circumstances as in Proposition \ref{fundtwist}, we can
express the virtual fundamental class $[\mgn(X,\beta)]\virt$ in
terms of $[\mgn(Y,\beta_Y)]\virt$ for suitable $\beta_Y\in
H_2(Y,\ZZ)$. In this case, the GW-invariants of the three-fold $X$
can be computed in terms of the GW-invariants of the surface $Y$.
In case $Y$ also admits a holomorphic two-form, for instance,
$Y=D\times E$ for an elliptic curve $E$, then we can apply our
localization by holomorphic two-form again and reduce the
computation further to the curve case. In case $Y$ is a ruled
surface over $D$, then by deforming $Y$ to $\PP (\cO\oplus L)$ for
some line bundle $L$ on $D$, we may assume $Y$ admits a torus
action. The virtual localization formula \cite{GraberPand} then can
be applied to this case to evaluate the GW-invariants of $Y$ in
terms of those of $D$.

Another important case is when $X$ is a ruled three-fold over a
surface $S$ with a canonical divisor $D$. In this case, we can
proceed as follows: using the circle action on the fibers of $X\to
S$ we can apply the virtual localization first and reduce the
computation of the invariants of $X$ to $S$. After that, we can
apply our localization principle by holomorphic two-form to further
reduce the computation to the curve case. As is indicated by our
degree two calculation of the GW-invariants of surfaces, further
works must be done to carry this through. Nevertheless, it is worth
pursuing due to the importance of the ruled three-folds in the
future investigation of GW-invariants of general three-folds.

\bibliographystyle{amsplain}

\begin{thebibliography}{10}

\bibitem{ACGH} E. Arbarello, M. Cornalba, P. Griffiths and J. Harris,
\textit{Geometry of algebraic curves}. Springer-Verlag, 1985.



\bibitem{Behrend} K. Behrend. \textit{Donaldson-Thomas invariants
via microlocal geometry.} math.AG/0507523.

\bibitem{BF} K. Behrend and B. Fantechi, \textit{The intrinsic
normal cone.} Invent. Math. \textbf{128} (1997), no. 1, 45--88.

\bibitem{Catanese} F.~Catanese, private communication.

\bibitem{CoatesGivental} T. Coates and A. Givental. \textit{Quantum
Riemann-Roch, Lefschetz and Serre.} Arxiv: math.AG/0110142.

\bibitem{FaberPand-HodgeInt} C. Faber and R. Pandharipande,
\textit{Hodge integrals and Gromov-Witten theory.} Invent. Math.
\textbf{139} (2000), no. 1, 173--199.

\bibitem{Fulton} W. Fulton. \textit{Intersection theory. Second edition.}
Ergebnisse der Mathematik und ihrer Grenzgebiete. 3. Folge.
Springer-Verlag, Berlin, 1998.

\bibitem{GraberPand} T. Graber and R. Pandharpande.
\textit{Hodge integrals and Gromov-Witten theory.} Invent. Math.
\textbf{139} (2000), no. 1, 173--199.

\bibitem{Harris} J. Harris, \emph{Theta characteristics on algebraic
curves.} Trans. Amer. Math. Soc. 271, no. 2 (1982) 611-638.

\bibitem{IonelParker} E. Ionel and T. Parker.
\textit{The symplectic sum formula for Gromov-Witten invariants.}
Ann. of Math. (2) \textbf{159} (2004), no. 3, 935--1025.

\bibitem{KL2} Y-H.~Kiem and J.~Li, in preparation.

\bibitem{KL3} Y-H.~Kiem and J.~Li, in preparation.

\bibitem{Kresch} A. Kresch.
\textit{Cycle groups for Artin stacks.} Invent. Math. \textbf{138}
(1999), no. 3, 495--536.

\bibitem{Lee} J. Lee. \textit{Holomorphic 2-forms and Vanishing Theorems for Gromov-Witten
Invariants.} Arxiv: math.SG/0610782.

\bibitem{Lee-Parker} J. Lee and T. Parker.
\textit{A Structure Theorem for the Gromov-Witten Invariants of
Kahler Surfaces.} Arxiv: math.SG/0610570. To appear in Jour. Diff.
Geom.

\bibitem{Li1} J. Li.
\textit{Stable morphisms to singular schemes and relative stable
morphisms.} J. Differential Geom. \textbf{57} (2001), no. 3,
509--578.

\bibitem{Li2} J. Li.
\textit{A degeneration formula of GW-invariants.} J. Differential
Geom. \textbf{60} (2002), no. 2, 199--293.


\bibitem{LL} J.~Li and W-P.~Li, \textit{Two point extremal
Gromov-Witten invariants of Hilbert schemes of points on surfaces.}
Arxiv: math.AG/0703717.

\bibitem{LiRuan} A-M.~Li and Y.~Ruan,
\textit{Symplectic surgery and Gromov-Witten invariants of
Calabi-Yau 3-folds}, Inv. Math. \textbf{145} (2001), no. 1, 151--218.

\bibitem{Li-Tian} J. Li and G. Tian.
\textit{Virtual moduli cycles and Gromov-Witten invariants of
algebraic varieties.} J. Amer. Math. Soc. \textbf{11} (1998), no. 1,
119--174.

\bibitem{LT2} J. Li and G. Tian,
\textit{Comparison of algebraic and symplectic Gromov-Witten
invariants}, Asian J. Math. \textbf{3} (1999), no. 3, 689--728.

\bibitem{MaulikPand} D. Maulik and R. Pandharipande.
\textit{New calculations in Gromov-Witten
theory.} Arxiv: math.AG/0601395.

\bibitem{OkounkovPand} A. Okounkov and R. Pandharipande.
\textit{Virasoro constraints for target curves.} Invent. Math.
\textbf{163} (2006), no. 1, 47--108.


\bibitem{BGP} A. Okounkov and R. Pandharipande. \textit{Quantum
cohomology of the Hilbert scheme of points in the plane.} Arxiv:
math.AG/0411210.


\bibitem{Radoux} C. Radoux. \textit{Nombres de Catalan
g\'en\'eralis\'es}. Bulletin of Belg. Math. Soc. \textbf{4} (1997),
no. 2, 289--292.

\end{thebibliography}

\end{document}